\documentclass[11pt]{amsart}

\usepackage{amssymb}
\usepackage{graphicx}
\usepackage[headings]{fullpage}
\usepackage{enumitem}
\usepackage{longtable}
\usepackage{hyperref}
\usepackage{url}
\usepackage{ytableau}
\usepackage{tikz}

\theoremstyle{definition}
\newtheorem{theorem}{Theorem}[section]
\newtheorem{proposition}[theorem]{Proposition}
\newtheorem{corollary}[theorem]{Corollary}
\newtheorem{lemma}[theorem]{Lemma}
\newtheorem{definition}[theorem]{Definition}
\newtheorem{remark}[theorem]{Remark}
\newtheorem{example}[theorem]{Example}
\newtheorem{problem}[theorem]{Problem}
\newtheorem*{thm}{Theorem}

\numberwithin{equation}{section}

\newcommand{\lex}{\mathrm{lex}}
\newcommand{\pa}{\operatorname{part}}
\newcommand{\sym}{\mathfrak{S}}
\renewcommand{\P}{\mathrm{P}}
\renewcommand{\geq}{\geqslant}
\renewcommand{\leq}{\leqslant}
\newcommand{\type}{\operatorname{type}}
\newcommand{\Mon}{\operatorname{Mon}}
\newcommand{\M}{\mathcal{M}}
\newcommand{\mideal}{\mathfrak{m}}
\newcommand{\Tor}{\operatorname{Tor}}

\author[BDGMNORS]{Jennifer Biermann}
\address{Department of Mathematics and Computer Science,
Hobart and William Smith Colleges,
300 Pulteney Street, Geneva, New York 14456, USA}
\email{\href{mailto:biermann@hws.edu}{\nolinkurl{biermann@hws.edu}}}

\author[]{Hern\'an de Alba}
\address{
CONACyT-UAZ, Unidad Acad\'emica de Matem\'aticas de la Universidad Aut\'onoma de Zacatecas,
Calzada Solidaridad entronque Paseo a la Bufa,
Zacatecas, Zac., CP 98000, Mexico}
\email{\href{mailto:hdealbaca@conacyt.mx}{\nolinkurl{hdealbaca@conacyt.mx}}}

\author[]{Federico Galetto}
\address{Department of Mathematics and Statistics, Cleveland State University, Cleveland, OH, 44115-2215, USA}
\email{\href{mailto:f.galetto@csuohio.edu}{\nolinkurl{f.galetto@csuohio.edu}}}

\author[]{Satoshi Murai}
\address{
Department of Mathematics, Faculty of Education,
Waseda University, 1-6-1 Nishi-Waseda, Shinjuku, Tokyo 169-8050, Japan}
\email{\href{mailto:s-murai@waseda.jp}{\nolinkurl{s-murai@waseda.jp}}}

\author[]{Uwe Nagel}
\address{Department of Mathematics, University of Kentucky, 715 Patterson office tower, Lexington, KY 40506-0027, USA}
\email{\href{mailto:uwe.nagel@uky.edu}{\nolinkurl{uwe.nagel@uky.edu}}}

\author[]{Augustine O'Keefe}
\address{Connecticut College, Mathematics Department, 270 Mohegan Avenue Pkwy, New London, CT 06320, USA}
\email{\href{mailto:aokeefe@conncoll.edu}{\nolinkurl{aokeefe@conncoll.edu}}}

\author[]{Tim R\"omer}
\address{Institut f\"ur Mathematik, Universit\"at Osnabr\"uck, 49069 Osnabr\"uck, Germany}
\email{\href{mailto:troemer@uos.de}{\nolinkurl{troemer@uos.de}}}

\author[]{Alexandra Seceleanu}
\address{Department of Mathematics, University of Nebraska-Lincoln,
203 Avery Hall, Lincoln NE 68588--0130, USA}
\email{\href{mailto:aseceleanu@unl.edu}{\nolinkurl{aseceleanu@unl.edu}}}

\title{Betti numbers of symmetric shifted ideals}



\begin{document}

\ytableausetup {mathmode, boxsize=0.8em}

\begin{abstract}
  We introduce a new class of monomial ideals which we call symmetric
  shifted ideals.  Symmetric shifted ideals are fixed by the natural
  action of the symmetric group and, within the class of monomial
  ideals fixed by this action, they can be considered as an analogue
  of stable monomial ideals within the class of monomial ideals.  We
  show that a symmetric shifted ideal has linear quotients and compute
  its (equivariant) graded Betti numbers.  As an application of this
  result, we obtain several consequences for graded Betti numbers of
  symbolic powers of defining ideals of star configurations.
\end{abstract}

\maketitle{}

\section{Introduction}
\label{sec:introduction}

The study of graded Betti numbers of graded ideals in a polynomial ring is one of the central topics in commutative algebra.  It has always been of great interest to find combinatorial formulas for these numbers for various families of monomial
ideals.  In this paper, we introduce a new class of monomial ideals,
which we call symmetric shifted ideals, and compute  their graded Betti
numbers.

For the definition of these ideals, we first provide some necessary notation. Let $S=\Bbbk[x_1,\dots,x_n]$ be a standard graded polynomial ring over a field $\Bbbk$. For
$\mathbf a =(a_1,\dots,a_n) \in \mathbb{Z}_{\geq 0}^n$, we write $x^{\mathbf a}=x_1^{a_1} \cdots x_n^{a_n}$ and
$|\mathbf a|=a_1+ \cdots +a_n$.  We consider an action of the symmetric group $\sym_n$ on $S$ defined by permutations of the variables and focus on $\sym_n$-fixed monomial ideals $I \subset S$, that is to say, monomial ideals $I \subset S$ with $\sigma(I)=I$ for all $\sigma \in \sym_n$. Such ideals have recently attracted attention
as elements of ascending chains of ideals that are invariant under actions of symmetric groups (see, e.g., \cite{AH07, Dr14, HS12, LNNR18a, LNNR18b, NR17}).

We say that a sequence $\lambda=(\lambda_1,\dots,\lambda_n)$ of
non-negative integers is a partition of $d$ of length $n$, if
$\lambda_1 \leq \cdots \leq \lambda_n$ and $|\mathbf \lambda|=d$.
Let
$$\P_n=\{(\lambda_1,\dots,\lambda_n) \in \mathbb{Z}^n:
0 \leq \lambda_1 \leq \lambda_2 \leq \cdots \leq \lambda_n\}$$ be the
set of partitions of length $n$.  For a monomial
$u=x_1^{a_1} \cdots x_n^{a_n}$ of degree $d$, we write
$\pa(u) \in \mathbb{Z}^n_{\geq 0}$ for the partition obtained from
$(a_1,\dots,a_n)$ by permuting its entries in a suitable way.  For example,
$\pa(x_1^2x_2^0x_3^1x_4^2)=(0,1,2,2)$.  If a monomial ideal
$I \subset S$ is $\sym_n$-fixed, then a monomial $u$ is in $I$ if and
only if $x^{\pa(u)}$ is in $I$.  Thus, the set
$$\P(I)=\{ \lambda \in \P_n:x^\lambda \in I\}$$
determines the monomials in $I$ and the ideal itself.  The central object of study of this note are:

\begin{definition}
  \label{def:shifted}
  Let $I \subset S$ be an $\sym_n$-fixed monomial ideal.  We say that
  $I$ is a \textbf{symmetric shifted ideal} if, for every
  $\lambda =(\lambda_1,\dots,\lambda_n) \in \P(I)$ and $1 \leq k <n$
  with $\lambda_k < \lambda_n$, one has $x^\lambda (x_k/x_n) \in I$.
  Also, we say that $I$ is a \textbf{symmetric strongly shifted ideal}
  if, for every $\lambda =(\lambda_1,\dots,\lambda_n) \in \P(I)$ and
  $1 \leq k <l \leq n$ with $\lambda_k < \lambda_l$, one has
  $x^\lambda (x_k/x_l) \in I$.  We may also refer to these ideals
  simply as shifted and strongly shifted ideals.
\end{definition}

In the following remarks, we note that the above properties can be
defined purely in terms of partitions.

\begin{remark}



  Let $I\subset S$ be an $\sym_n$-fixed monomial ideal. Denote by
  $e_i$ the $i$-th standard basis vector of $\mathbb{Z}^n$.  If for
  every $\lambda=(\lambda_1,\dots,\lambda_n ) \in \P (I)$ with
  $j=\min\{k:\lambda_k=\lambda_n\}$ and for every
  $\lambda - e_j + e_i$ with $i<j$ that is also a partition (i.e., non-decreasing) we have
  $\lambda - e_j + e_i \in \P (I)$, then $I$ is shifted.
\end{remark}

\begin{remark}
  For partitions
  $\lambda=(\lambda_1,\dots,\lambda_n),\mu=(\mu_1,\dots,\mu_n)$, we
  define
  $$\mu \unlhd \lambda \text{ if }
  \mu_k+ \cdots +\mu_n \leq \lambda_k + \cdots + \lambda_n \text{ for all }  k.$$
  The partial order $\lhd$ is known in the literature as the
  \textbf{dominance order} (see, e.g., \cite[\S{}7.2]{Stanley}). An $\sym_n$-fixed monomial ideal $I$ is
  strongly shifted if and only if, for every $\lambda, \mu \in \P_n$
  with $|\lambda|=|\mu|$, $\lambda \in P(I)$ and $\mu \unlhd \lambda $
  imply $\mu \in \P(I)$.
\end{remark}

The definition of shifted and strongly shifted ideals is inspired by
the definition of stable and strongly stable ideals, which
are important classes of monomial ideals since, e.g.,
in characteristic zero generic initial initials are strongly stable.
Recall that Eliahou and Kervaire \cite{EK} constructed minimal graded free resolutions of stable
ideals and gave a simple formula for their graded Betti numbers in terms of
the data of their minimal systems of monomial generators. The
main results of this paper are the following formulas for graded Betti
numbers of symmetric shifted ideals.

\begin{itemize}
\item[(1)] We prove that every symmetric shifted ideal $I$ has linear
  quotients (Theorem \ref{thm:1}). This allows us to give a formula
  for its graded Betti numbers in terms of its monomial generators
  $G(I)$ (Corollary \ref{cor:betti_shifted}).
\item[(2)] We also give a formula for the graded Betti numbers of a
  symmetric shifted ideal $I$ in terms of its partition generators
  $\{\lambda \in P_n: x^\lambda \in G(I)\}$ (Corollary \ref{2.5}).
\item[(3)] We compute equivariant graded Betti numbers of a symmetric
  shifted ideal $I$. In other words, we determine the
  $\Bbbk[\sym_n]$-module structure of $\Tor_i(I,\Bbbk)_j$ (Theorem
  \ref{3.2}).
\end{itemize}

Our initial motivation for defining symmetric shifted ideals comes from the
study of minimal graded free resolutions of symbolic powers of star
configurations.  A \textbf{codimension $c$ star configuration} is a union
of linear subspaces of a projective space $\mathbb P^N$ of the form
$$V_c=\bigcup_{1 \leq i_1 < \cdots < i_c \leq n} H_{i_1}\cap \dots \cap H_{i_c},$$
where $H_1,\dots,H_n$ are distinct hyperplanes in $\mathbb P^N$ such
that the intersection of any $j$ of them is either empty or has
codimension $j$ and where $1 \leq c \leq \min\{n,N\}$.  The name is motivated by the special case of 10 points located at pairwise intersections
of 5 lines in the projective plane, with the lines positioned in the
shape of a star.  Let $L_1,\dots,L_n$ be defining linear forms of
$H_1,\dots,H_n$. Then the defining ideal of $V_c$ is given by
$I_{V_c}=\bigcap_{1 \leq i_1 < \cdots < i_c \leq n}
(L_{i_1},\dots,L_{i_c})$ and its $m$-th symbolic power can be written
as
$$I_{V_c}^{(m)}=\bigcap_{1 \leq i_1 < \cdots < i_c \leq n} (L_{i_1},\dots,L_{i_c})^m$$
because each ideal $(L_{i_1},\dots,L_{i_c})$ is a minimal prime of
$I_{V_c}$ and it is generated by a regular sequence; see, e.g., \cite[Appendix 6,
Lemma 5]{ZS}.  These ideals $I_{V_c}$ and $I_{V_c}^{(m)}$ have been
extensively studied from the point of view of algebra, geometry, and
combinatorics in
\cite{AS12,AS14,BDRH+,BH10,CCG08,CGVT14,CHT11,Fra08,FMN06,GMS06,MN05,PS15,Sch82,TVAV11,Var11}.
We recommend \cite{GHM} as a great introduction to the subject.
In
particular, the Betti numbers of the defining ideal of a star
configuration and its symbolic square have been determined in \cite[Remark 2.11,
Theorem 3.2]{GHM}.
Further motivation for studying these ideals can be found in \cite{GHMN}, which considers generalizations where the linear forms are replaced by forms of arbitrary degree and also explores connections with Stanley-Reisner ideals of matroids.

Let $I_{n,c}$ be
the monomial ideal of $S = \Bbbk [x_1,\dots,x_n]$ defined by
\begin{equation*}
  I_{n,c} = \bigcap_{1\leq i_1 < \dots < i_c \leq n}
  (x_{i_1}, \dots, x_{i_c}).
\end{equation*}
Then the minimal graded free resolution of $I_{V_c}^{(m)}$ is
completely determined by that of $I_{n,c}^{(m)}$; in particular, it
was shown that these two ideals have the same graded Betti numbers
(see \cite[Example 3.4 and Theorem 3.6]{GHMN}, where the more general
case of hypersurface or matroid configurations is considered). The
same reference also shows that both ideals are Cohen-Macaulay.  Note
that the ideal $I_{n,c}$ can also be described as the ideal generated
by all squarefree monomials of degree $n -c + 1$ \cite[Proposition
2.3]{GHMN}.  Obviously the ideal $I_{n,c}^{(m)}$ is $\sym_n$-fixed.
As one of our main results we prove (in Theorem \ref{thm:mainresult})
that:

\begin{thm}
The ideal $I_{n,c}^{(m)}$ is strongly shifted.
\end{thm}

Since we find a formula for the graded Betti numbers of symmetric
shifted ideals, this result gives various information on graded Betti
numbers of symbolic powers of star configurations, including their
Castelnuovo-Mumford regularity, a simple formula for the Betti numbers
in the top and bottom rows of the Betti table, an explicit formula for
the Betti numbers of the symbolic cube, and more (see the results in
Section \ref{sec:star-configurations}). Our results for star
configurations also apply to the computation of Betti numbers of fat
point schemes (see Remark \ref{rem:fat_points}).

This paper is organized as follows: In Section
\ref{sec:shifted-ideals}, we study some combinatorial properties of
symmetric shifted ideals.  In Section \ref{sec:shifted-ideals-have},
we prove that symmetric shifted ideals have linear quotients, and in
Section \ref{sec:star-configurations} we apply the results in Section
\ref{sec:shifted-ideals-have} to symbolic powers of star
configurations.  In Sections \ref{sec:decomp-symm-shift} and
\ref{sec:equiv-betti-numb}, we compute the (equivariant) graded Betti
numbers of symmetric shifted ideals.

\subsection*{Acknowledgments} This work was started at the workshop
``Ordinary and Symbolic Powers of Ideals'' at Casa Matemática Oaxaca
(CMO) in May 2017. We thank the organizers of the workshop and CMO for
their kind invitation and warm hospitality.

We also thank the anonymous referee who provided several useful
suggestions to clarify our exposition and improve the quality of our
manuscript.

Shortly after this paper was posted on arXiv, another preprint
appeared by Paolo Mantero \cite{1907.08172} which also computes the
graded Betti numbers of symbolic powers of star configurations. The
results in Mantero's preprint were obtained independently from ours,
and utilize new and interesting techniques. We are also grateful to
Paolo for pointing out a mistake in an earlier version of Corollary
\ref{cor:partial_betti_star}. In December 2019, another preprint
appeared on arXiv by Kuei-Nuan Lin and Yi-Huang Shen that uses and
generalizes some of the results in our paper to $a$-fold product
ideals \cite{LS19}.

The research of the fourth author is partially supported by KAKENHI
16K05102. The fifth author was partially supported by Simons
Foundation grant \#317096. The last author was supported by NSF grant
DMS–1601024 and EpSCOR award OIA–1557417.

\section{Symmetric shifted ideals}
\label{sec:shifted-ideals}

In this section, we discuss some basic properties of symmetric shifted ideals.
For a monomial $u \in S$, we write $\sym_n \cdot u$ for the
$\sym_n$-orbit of $u$ in $S$, i.e.,
$\sym_n \cdot u=\{\sigma(u): \sigma \in \sym_n\}$.  The set $\P_n$ can
be regarded as a poset with the order defined by $\lambda \geq \mu$ if
$x^\mu$ divides $x^\lambda$.  Then the set $\P(I)$ is a filter in the
poset $\P_n$, that is to say, for $\mu \in \P(I)$ and
$\lambda \in \P_n$, one has $\lambda \in \P(I)$ if $\lambda \geq \mu$.
The next lemma shows that the assignment $I \mapsto \P(I)$ defines a
one-to-one correspondence between $\sym_n$-fixed monomial ideals in
$S$ and filters in $\P_n$.

\begin{lemma}
\label{1.1}
Let $\lambda,\mu \in \P_n$.  There exist monomials
$u \in \sym_n \cdot x^\mu$ and $w \in \sym_n \cdot x^\lambda$ such
that $u$ divides $w$ if and only if $x^\mu$ divides $x^\lambda$.
\end{lemma}

\begin{proof}
  The ``if'' part is obvious. We prove the ``only if'' part.  Let
  $\lambda=(\lambda_1,\dots,\lambda_n)$ and $\mu=(\mu_1,\dots,\mu_n)$.
  The assumption says that there exists $\sigma \in \sym_n$ such that
  $$\mu_{\sigma(1)} \leq \lambda_1, \dots,\mu_{\sigma(n)} \leq \lambda_n.$$
  Since $\lambda_1 \leq \cdots \leq \lambda_n$, for each
  $k=1,2,\dots,n$ we have
  $$\mu_{\sigma(1)} \leq \lambda_k,\mu_{\sigma(2)} \leq \lambda_k,\dots,\mu_{\sigma(k)}\leq \lambda_k.$$
  This implies that the partition $\mu$ contains at least $k$ entries
  smaller than or equal to $\lambda_k$.  Therefore
  $\mu_k \leq \lambda_k$ for all $k$, and $x^\mu$ divides $x^\lambda$.
\end{proof}

Throughout the rest of the paper, we will say that $\mu \in \P_n$
divides $\lambda \in \P_n$ if $x^\mu$ divides $x^\lambda$.

Next, we show that to check the conditions of symmetric (strongly)
shifted ideals, it is enough to check them on generators.  Let $I$ be a
monomial ideal.  We write $G(I)$ for the unique set of minimal
monomial generators of $I$.  When $I$ is $\sym_n$-fixed, we define
$$\Lambda(I)=\{ \lambda \in P(I): x^\lambda \in G(I)\}.$$
Note that
$G(I)=\biguplus_{\lambda \in \Lambda(I)} \sym_n\cdot x^\lambda$, where
$\biguplus$ denotes a disjoint union of sets.  As the next statement
shows, to check $I$ is shifted it is enough to check the condition of
Definition \ref{def:shifted} for partitions in $\Lambda(I)$.

\begin{lemma}
  \label{1.2}
  Let $I \subset S$ be an $\sym_n$-fixed monomial ideal.  Then $I$ is
  shifted if and only if, for every
  $\lambda=(\lambda_1,\dots,\lambda_n) \in \Lambda(I)$ and
  $1 \leq k<n$ with $\lambda_k < \lambda_n$, one has
  $x^\lambda ( x_k/x_n) \in I$.
\end{lemma}

\begin{proof}
  The ``only if'' part is obvious.  We prove the ``if'' part.  Let
  $\mu=(\mu_1,\dots,\mu_n) \in \P(I)$ and $1 \leq k <n$
  with $\mu_k < \mu_n$.  We claim $x^\mu (x_k/x_n) \in I$.

  Let $\lambda=(\lambda_1,\dots,\lambda_n) \in \Lambda(I)$ be a partition that
  divides $\mu$.  If $\lambda_n=\mu_n$, then
  $w=x^\lambda (x_k/x_n) \in I$ by assumption and $w$ divides
  $x^\mu (x_k/x_n)$.  If $\lambda_n <\mu_n$, then $x^\lambda x_k$
  divides $x^\mu (x_k/x_n)$.  In both cases,
  $x^\mu(x_k/x_n) \in I$ as desired.
\end{proof}

An analogous statement holds for symmetric strongly shifted ideals. We
omit the proof since it is essentially the same as the one for
symmetric shifted ideals.

\begin{lemma}
  \label{1.2bis}
  Let $I \subset S$ be an $\sym_n$-fixed monomial ideal.  Then $I$ is
  strongly shifted if and only if, for every
  $\lambda=(\lambda_1,\dots,\lambda_n) \in \Lambda(I)$ and
  $1 \leq k < l \leq n$ with $\lambda_k < \lambda_l$, one has
  $x^\lambda ( x_k/x_l) \in I$.
\end{lemma}

\begin{example}
  Let $I \subset \Bbbk[x_1,x_2,x_3]$ be the $\sym_3$-fixed monomial
  ideal with
  \begin{equation*}
    \Lambda(I)=\{(1,1,1), (0,1,2), (0,0,4)\}.
  \end{equation*}
  The ideal $I$ is strongly shifted and is minimally generated by the
  following ten monomials:
  $$x_1x_2x_3,\quad x_2x_3^2,\quad x_2^2x_3,\quad x_1x_3^2,\quad x_1^2x_3,\quad
  x_1x_2^2,\quad x_1^2x_2,\quad x_1^4,\quad x_2^4,\quad x_3^4.$$
\end{example}

\begin{example}
  Let $I \subset \Bbbk[x_1,x_2,x_3,x_4]$ be the $\sym_4$-fixed
  monomial ideal with
  \begin{equation*}
    \Lambda(I)=\{(1,1,2,2),(0,2,2,2),(0,1,2,3)\}.
  \end{equation*}
  Then $I$ is shifted but not strongly shifted since
  $(0,1,2,3) \in \P(I)$ but $(1,1,1,3) \not \in \P(I)$.
\end{example}

For a partition $\lambda=(\lambda_1,\dots,\lambda_n)$, we define the
quantities
\begin{equation*}
  \begin{split}
    &p(\lambda)=\#\{k: \lambda_k < \lambda_n-1\},\\
    &r(\lambda)=\#\{k: \lambda_k=\lambda_n\}.
  \end{split}
\end{equation*}
We also introduce the truncation of the partition $\lambda$ by setting
$$\lambda_{\leq k}=(\lambda_1,\dots,\lambda_k)$$
for $k=1,2,\dots,n$.  Let $<_\lex$ be the total order on
$\mathbb{Z}_{\geq 0}^n$ defined by
$$\mathbf a=(a_1,\dots,a_n) <_\lex (b_1,\dots,b_n)=\mathbf b$$
if
\begin{enumerate}[label=(\roman*)]
\item $|\mathbf a| < |\mathbf b|$, or
\item $|\mathbf a| = |\mathbf b|$ and the leftmost non-zero entry of
  $(a_1-b_1,\dots,a_n-b_n)$ is positive.
\end{enumerate}

\begin{remark}
  Our definition of the order $<_\lex$ is the opposite of the more
  familiar lexicographic order for monomials (cf. \cite[Ch.~2
  \S{}2]{CLO}). This is necessary to ensure compatibility with our
  definition of partitions as non-decreasing sequences.
\end{remark}

We establish another result that will be used in later sections.

\begin{lemma}
  \label{1.3}
  Let $I$ be a symmetric shifted ideal.  For every
  $\mu \in P(I)$, there is a unique $\lambda \in \Lambda(I)$ such that
  \begin{enumerate}[label=(\alph*)]
  \item $\lambda$ divides $\mu$, and
  \item $\lambda_{\leq p(\lambda)}=\mu_{\leq p(\lambda)}$.
  \end{enumerate}
\end{lemma}

\begin{proof}
  Let $\mu=(\mu_1,\dots,\mu_n) \in \P(I)$.  Let
  $$\lambda=(\lambda_1,\dots,\lambda_n)=
  \min_{<_\lex} \{ \rho \in \Lambda(I): \rho \mbox{ divides } \mu\}$$
  and $p=p(\lambda)$.  Clearly $\lambda$ satisfies condition (a).   We claim that $\lambda$ fulfills also (b), that is to say, $\lambda_k=\mu_k$ for all $k \leq p$.

  Suppose to the contrary that there is $k \leq p$ such that
  $\lambda_k < \mu_k$.  Then $w=x^\lambda(x_k/x_n)$ divides $x^\mu$
  and, by definition of symmetric shifted ideals, we have $w \in I$.
  Let $\lambda'= \pa(w)$.  Observe that $\lambda'$ is constructed from
  $\lambda$ by replacing the part $\lambda_n$ with $\lambda_n -1$, the
  part $\lambda_k$ with $\lambda_k +1$, and rearranging in
  non-decreasing order. By definition of $r=r(\lambda)$, the partition
  $\lambda$ has $r$ parts equal to $\lambda_n$. Now suppose that
  $\lambda' = \lambda$. Then $\lambda'$ also has $r$ parts equal to
  $\lambda_n$, so we must have $\lambda_k +1 = \lambda_n$ or
  $\lambda_k = \lambda_n -1$. However, the definition of
  $p = p(\lambda)$ implies that $\lambda_k < \lambda_n -1$. We deduce
  that $\lambda' \neq \lambda$.  Then ${\lambda'}$ divides $\mu$ by
  Lemma \ref{1.1} and there is a $\rho \in \Lambda(I)$ that divides
  $\lambda'\in \P(I)$. However, such a $\rho$ satisfies
  $\rho \leq_\lex \lambda' <_\lex \lambda$, contradicting the
  minimality of $\lambda$.

  Next, we prove uniqueness.  Suppose that
  $\lambda,\lambda' \in \Lambda(I)$ satisfy conditions (a) and (b).
  We prove $\lambda=\lambda'$.  Let $p=p(\lambda)$ and
  $p'=p(\lambda')$.  We may assume $p \leq p'$.  By condition (b),
  $\lambda$ and $\lambda'$ are of the form
  $$\lambda=(\mu_1,\dots,\mu_p,
  \lambda_n-1,\dots,\lambda_n-1,\lambda_n,\dots,\lambda_n)$$
  and
  $$\lambda'=(\mu_1,\dots,\mu_p,\mu_{p+1}\dots,\mu_{p'},
  \lambda'_n-1,\dots,\lambda'_n-1,\lambda'_n,\dots,\lambda'_n).$$
  Suppose $p<p'$.  Since $\lambda$ divides $\mu$, we have
  $$\mu_k \geq \lambda_k \geq \lambda_n-1 \ \ \ \mbox{ for }p<k\leq p',$$
  and
  $$\lambda'_n -1 > \mu_{p'} \geq \lambda_n-1.$$
  But these inequalities say that $\lambda$ properly divides
  $\lambda'$, contradicting $\lambda,\lambda' \in \Lambda(I)$.  Hence
  $p=p'$.  However, given the shape of $\lambda$ and $\lambda'$,
  $p=p'$ implies that either $\lambda$ divides $\lambda'$ or
  $\lambda'$ divides $\lambda$.  Since
  $\lambda,\lambda' \in \Lambda(I)$, $\lambda$ must be equal to
  $\lambda'$ as desired.
\end{proof}

\section{Symmetric shifted ideals have linear quotients}
\label{sec:shifted-ideals-have}

A monomial ideal $I \subset S$ is said to have \textbf{linear quotients}
if there is an order $u_1,\dots,u_s$ of monomials in $G(I)$ such that
the colon ideal
$$(u_1,\dots,u_{k-1}):u_k$$
is generated by variables for all $k=2,3,\dots,s$.  A nice consequence
of having linear quotients is that we can easily compute the graded
Betti numbers from the above colon ideals.  Recall that, for a
graded ideal $I \subset S$, graded Betti numbers of $I$ are the
numbers $\beta_{i,j}(I)=\dim_\Bbbk\Tor_i(I,\Bbbk)_j$.  Herzog and Takayama
produced a formula for the bigraded Poincar\'e series of a monomial
ideal with linear quotients \cite[Corollary 1.6]{HT}:

\begin{theorem}\label{HerTaka}
With the same notation as above,
$$\beta_{i,i+j}(I) = \sum_{ \deg (u_k)=j} \binom{ |G((u_1,\dots,u_{k-1}):u_k)|}{i}.$$
\end{theorem}

Next, we present our first main result about symmetric shifted ideals.

\begin{theorem}
  \label{thm:1}
  Symmetric shifted ideals have linear quotients.
\end{theorem}

Using the same notation as in Section \ref{sec:shifted-ideals}, we
define a total order on the set of monomials in $S$.
Let $\lambda,\mu \in \P_n$. For distinct
monomials $v= \tau(x^\mu)$ and $u=\sigma(x^\lambda)$ in $S$, we define
$v \prec u$ if
\begin{enumerate}[label=(\roman*)]
\item $ \mu <_\lex \lambda$, or
\item $\mu=\lambda$ and $v <_\lex u$.
\end{enumerate}
Note in particular that if $v$ strictly divides $u$, then $v \prec u$.

\begin{proof}[Proof of Theorem \ref{thm:1}.]
  Let $I \subset S$ be a symmetric shifted ideal
  and fix a monomial $u=\sigma(x^\lambda) \in G(I)$ with
  $\lambda=(\lambda_1,\dots,\lambda_n) \in \Lambda(I)$.  Let
  $p=p(\lambda)$ and $r=r(\lambda)$. Thus, we have
  $$u=\sigma(x^\lambda)
  = x_{\sigma(1)}^{\lambda_1} \cdots x_{\sigma(p)}^{\lambda_{p}}
  x_{\sigma(p+1)}^{\lambda_n-1} \cdots x_{\sigma(n-r)}^{\lambda_n-1}
  x_{\sigma(n-r+1)}^{\lambda_n} \cdots x_{\sigma(n)}^{\lambda_n}.$$
  We also define the quantity
  $$\max(u)=\max\{\sigma(k):\lambda_k=\lambda_n\}
  =\max\{\sigma(n-r+1),\dots,\sigma(n)\}.$$
  Let $$J=(v \in G(I): v \prec u).$$
  We claim that
  \begin{equation}
    \label{eq_2.1}
    J:u=(x_{\sigma(1)},\dots,x_{\sigma(p)})
    +(x_{\sigma(k)}: p+1 \leq k \leq n-r,\ \sigma(k)<\max(u)).
  \end{equation}
  This proves that $I$ has linear quotients.

  We prove the containment ``$\supseteq$'' holds in \eqref{eq_2.1}.
  We first prove $x_{\sigma(k)} u \in J$ for $1 \leq k \leq p$.  In
  this case we have $\lambda_k < \lambda_n -1 < \lambda_n$.  Together
  with the fact that $u\in I$, this implies
  $x^\lambda (x_k/x_n) \in I$ because $I$ is shifted. It follows that
  the monomial
  $$w = u (x_{\sigma(k)}/x_{\sigma(n)}) =
  \sigma (x^\lambda (x_k/x_n))$$ is also in $I$ because $I$ is
  $\sym_n$-fixed. Reasoning as in the existence part of Lemma \ref{1.3}, we have $\pa(w) <_\lex \lambda$, so
  $w \prec u$.  This implies $w\in J$. Therefore we have
  $ x_{\sigma(k)} u =x_{\sigma (n)} w \in J$.

  Next we prove $x_{\sigma (k)} u \in J$ whenever
  $p+1 \leq k \leq n-r$ and $\sigma(k) < \max(u)$.  In this case, the
  monomial $w=u (x_{\sigma(k)}/x_{\max(u)}) \in I$ is obtained from $u$ by
  permuting variables so $\pa (w) = \lambda$. However,
  $\sigma(k) < \max(u)$ implies $w <_\lex u$.  Again we have
  $w\prec u$ and $w\in J$, so $x_{\sigma(k)} u =x_{\max (u)} w \in J$
  as desired.

  We prove the containment ``$\subseteq$'' holds in \eqref{eq_2.1}.
  Let $z \in S$ be a monomial not divisible by any variable in the set
  $$\{x_{\sigma(1)},\dots,x_{\sigma(p)}\}
  \cup \{x_{\sigma(k)}: p+1 \leq k \leq n-r,\ \sigma(k) < \max(u)\}.$$
  We can write $z u$ as
  \begin{equation}
    \label{A}
    z u
    = x_{\sigma(1)}^{\lambda_1} \cdots x_{\sigma(p)}^{\lambda_{p}}
    x_{\sigma(p+1)}^{b_{p+1}} \cdots x_{\sigma(n-r)}^{b_{n-r}}
    x_{\sigma(n-r+1)}^{b_{n-r+1}} \cdots x_{\sigma(n)}^{b_n},
  \end{equation}
  where $b_i \geq \lambda_i \geq \lambda_n-1$ for all $i \geq p+1$ and
  $b_i= \lambda_n-1$ for all $p+1\leq i \leq n-r$ with
  $\sigma(i) < \max(u)$.  We must prove that $z u \notin J$.

  Assume, by contradiction, that $z u \in J$.  By Equation \eqref{A},
  we have that:
  \begin{enumerate}[label=(\alph*)]
  \item\label{item:1} $\lambda$ divides $\pa (zu)$, and
  \item\label{item:2} $\lambda_{\leq p} = \pa (zu)_{\leq p}$.
  \end{enumerate}
  Since $\lambda = \pa (u) \in \Lambda (I)$, Lemma \ref{1.3}
  guarantees that $\lambda$ is the unique partition in $\Lambda (I)$
  satisfying properties \ref{item:1} and \ref{item:2}. Now define the
  ideal $$J'=(v \in G(I): \pa(v)<_\lex \lambda).$$  Note that $J'$ is
  $\sym_n$-fixed and shifted because $I$ is. Moreover
  $\Lambda (J') \subset \Lambda (I)$ and
  $\lambda \in \Lambda (I) \setminus \Lambda (J')$. Hence,
  $zu \notin J'$ by Lemma \ref{1.3}.

  Since $zu \in J$, there is a
  monomial $w \in G(J)$ that divides $z u$. Because $zu \notin J'$, we
  deduce that $w\notin G(J')$ so $\pa (w) \geq_\lex \lambda$. At the
  same time, $w\in G(J)$ gives $w \prec u$, so
  $\pa (w) \leq_\lex \lambda$. This forces $\pa (w) = \lambda$,
  therefore $w = \tau (x^\lambda)$ for some $\tau \in \sym_n$. More
  explicitly, $w$ is of the form
  $$w=x_{\tau(1)}^{\lambda_1} \cdots x_{\tau(p)}^{\lambda_{p}}
  x_{\tau(p+1)}^{\lambda_n-1} \cdots x_{\tau(n-r)}^{\lambda_n-1}
  x_{\tau(n-r+1)}^{\lambda_{n}} \cdots x_{\tau(n)}^{\lambda_n}.$$
  Comparing with Equation \eqref{A}, we get
  \begin{equation}
    \label{Eq_2.3}
    \{\sigma(1),\dots,\sigma(p)\}=\{\tau(1),\dots,\tau(p)\}.
  \end{equation}
  Observe that $\lambda_{\sigma^{-1}(k)}$ and $\lambda_{\tau^{-1}(k)}$
  are the exponents of $x_k$ in $u=\sigma(x^\lambda)$ and
  $w=\tau(x^\lambda)$, respectively. If
  $\lambda_{\sigma^{-1}(k)} = \lambda_{\tau^{-1}(k)}$ for all
  $1\leq k\leq n$, then $u=w$, which contradicts $w\prec u$.
  Therefore it makes sense to define
  $$\ell=\min\{k: \lambda_{\tau^{-1}(k)} \ne \lambda_{\sigma^{-1}(k)}\},$$
  and to write $\ell =\sigma(q)$ for some $1\leq q\leq n$.  By
  Equation \eqref{Eq_2.3}, we have
  $$\{\lambda_{\sigma^{-1}(\ell)},\lambda_{\tau^{-1}(\ell)}\}
  =\{\lambda_n-1,\lambda_n\}.$$ Since
  $w=\tau(x^\lambda) <_\lex \sigma(x^\lambda)=u$ by the definition of
  $\prec$, we actually have $\lambda_{\tau^{-1}(\ell)}=\lambda_n$ and
  $\lambda_q=\lambda_{\sigma^{-1}(\ell)}=\lambda_n-1$.  Also, since
  $w \ne u$, there is $m > \ell$ such that
  $\lambda_{\tau^{-1}(m)} =\lambda_n-1$ and
  $\lambda_{\sigma^{-1}(m)}=\lambda_n$. This shows that
  $$\sigma(q)=\ell < m = \sigma (\sigma^{-1} (m)) \leq\max(u),$$
  and therefore $b_q=\lambda_n-1$.  However, this contradicts the fact
  that $w$ divides $z u$ since the exponent of $x_\ell$ in $w$ is
  $\lambda_n$ but the exponent of $x_\ell=x_{\sigma(q)}$ in $z u$ is
  $b_q=\lambda_n-1$.
\end{proof}

\begin{example}
  \label{exa:shifted_lin_quot}
  Let $I \subset \Bbbk[x_1,x_2,x_3,x_4]$ be the symmetric shifted ideal with
  $$\Lambda(I)=\{(1,1,2,2),(0,2,2,2),(0,1,2,3)\}.$$ The ideal $I$ has 34
  generators. We arrange them in an increasing sequence using the
  order $\prec$ described at the beginning of this section, and we
  denote them $u_1,\dots,u_{34}$. We also set
  $I_{i-1} = (u_1,\dots,u_{i-1})$ for $2\leq i\leq 34$. Table
  \ref{tab:lin_quot_ex} shows all the linear quotients $I_{i-1}:(u_i)$
  of the ideal $I$ in the given order of the generators. All
  computations were performed using Macaulay2 \cite{M2}.

  \begin{table}[!htb]
  \caption{Linear quotients of a symmetric shifted ideal}
    \label{tab:lin_quot_ex}
  \begin{minipage}[t]{0.45\linewidth}
    \renewcommand{\arraystretch}{1.2}
    \centering
    \begin{tabular}{cccc}
      $i$ & $u_i$ & $I_{i-1}:(u_i)$ & $\max(u_i)$\\ \hline\hline
      1 & ${x}_{1}^{2} {x}_{2}^{2} {x}_{3} {x}_{4}$ & - & $2$\\ \hline
      2 & ${x}_{1}^{2} {x}_{2} {x}_{3}^{2} {x}_{4}$ & (${x_2}$) & $3$\\ \hline
      3 & ${x}_{1}^{2} {x}_{2} {x}_{3} {x}_{4}^{2}$ & (${x_2,x_3}$) & $4$\\ \hline
      4 & ${x}_{1} {x}_{2}^{2} {x}_{3}^{2} {x}_{4}$ & (${x_1}$) & $3$\\ \hline
      5 & ${x}_{1} {x}_{2}^{2} {x}_{3} {x}_{4}^{2}$ & (${x_1, x_3}$) & $4$\\ \hline
      6 & ${x}_{1} {x}_{2} {x}_{3}^{2} {x}_{4}^{2}$ & (${x_1, x_2}$) & $4$\\ \hline
      7 & ${x}_{1}^{2} {x}_{2}^{2} {x}_{3}^{2}$ & (${x_4}$) & $3$\\ \hline
      8 & ${x}_{1}^{2} {x}_{2}^{2} {x}_{4}^{2}$ & (${x_3}$) & $4$\\ \hline
      9 & ${x}_{1}^{2} {x}_{3}^{2} {x}_{4}^{2}$ & (${x_2}$) & $4$\\ \hline
      10 & ${x}_{2}^{2} {x}_{3}^{2} {x}_{4}^{2}$ & (${x_1}$) & $4$\\ \hline
      11 & ${x}_{1}^{3} {x}_{2}^{2} {x}_{3}$ & (${x_3, x_4}$) & $1$\\ \hline
      12 & ${x}_{1}^{3} {x}_{2}^{2} {x}_{4}$ & (${x_3, x_4}$) & $1$\\ \hline
      13 & ${x}_{1}^{3} {x}_{2} {x}_{3}^{2}$ & (${x_2, x_4}$) & $1$\\ \hline
      14 & ${x}_{1}^{3} {x}_{2} {x}_{4}^{2}$ & (${x_2, x_3}$) & $1$\\ \hline
      15 & ${x}_{1}^{3} {x}_{3}^{2} {x}_{4}$ & (${x_2, x_4}$) & $1$\\ \hline
      16 & ${x}_{1}^{3} {x}_{3} {x}_{4}^{2}$ & (${x_2, x_3}$) & $1$\\ \hline
      17 & ${x}_{1}^{2} {x}_{2}^{3} {x}_{3}$ & (${x_1, x_3, x_4}$) & $2$\\ \hline
    \end{tabular}
  \end{minipage}
  \begin{minipage}[t]{0.45\textwidth}
    \renewcommand{\arraystretch}{1.2}
    \centering
    \begin{tabular}{cccc}
      $i$ & $u_i$ & $I_{i-1}:(u_i)$ & $\max(u_i)$\\ \hline\hline
    18 & ${x}_{1}^{2} {x}_{2}^{3} {x}_{4}$ & (${x_1, x_3, x_4}$) & $2$\\ \hline
      19 & ${x}_{1}^{2} {x}_{2} {x}_{3}^{3}$ & (${x_1, x_2, x_4}$) & $3$\\ \hline
      20 & ${x}_{1}^{2} {x}_{2} {x}_{4}^{3}$ & (${x_1, x_2, x_3}$) & $4$\\ \hline
      21 & ${x}_{1}^{2} {x}_{3}^{3} {x}_{4}$ & (${x_1, x_2, x_4}$) & $3$\\ \hline
      22 & ${x}_{1}^{2} {x}_{3} {x}_{4}^{3}$ & (${x_1, x_2, x_3}$) & $4$\\ \hline
      23 & ${x}_{1} {x}_{2}^{3} {x}_{3}^{2}$ & (${x_1, x_4}$) & $2$\\ \hline
      24 & ${x}_{1} {x}_{2}^{3} {x}_{4}^{2}$ & (${x_1, x_3}$) & $2$\\ \hline
      25 & ${x}_{1} {x}_{2}^{2} {x}_{3}^{3}$ & (${x_1, x_2, x_4}$) & $3$\\ \hline
      26 & ${x}_{1} {x}_{2}^{2} {x}_{4}^{3}$ & (${x_1, x_2, x_3}$) & $4$\\ \hline
      27 & ${x}_{1} {x}_{3}^{3} {x}_{4}^{2}$ & (${x_1, x_2}$) & $3$\\ \hline
      28 & ${x}_{1} {x}_{3}^{2} {x}_{4}^{3}$ & (${x_1, x_2, x_3}$) & $4$\\ \hline
      29 & ${x}_{2}^{3} {x}_{3}^{2} {x}_{4}$ & (${x_1, x_4}$) & $2$\\ \hline
      30 & ${x}_{2}^{3} {x}_{3} {x}_{4}^{2}$ & (${x_1, x_3}$) & $2$\\ \hline
      31 & ${x}_{2}^{2} {x}_{3}^{3} {x}_{4}$ & (${x_1, x_2, x_4}$) & $3$\\ \hline
      32 & ${x}_{2}^{2} {x}_{3} {x}_{4}^{3}$ & (${x_1, x_2, x_3}$) & $4$\\ \hline
      33 & ${x}_{2} {x}_{3}^{3} {x}_{4}^{2}$ & (${x_1, x_2}$) & $3$\\ \hline
      34 & ${x}_{2} {x}_{3}^{2} {x}_{4}^{3}$ & (${x_1, x_2, x_3}$) & $4$\\ \hline
    \end{tabular}
  \end{minipage}
\end{table}
\end{example}

Using these results together with Theorem \ref{thm:1}, we can give a
formula for graded Betti numbers of symmetric shifted ideals. For a
monomial $u=\sigma(x^\lambda)$, recall our earlier notations:
\begin{equation*}
  \begin{split}
    &p=p(\lambda)=\#\{k: \lambda_k < \lambda_n-1\},\\
    &r=r(\lambda)=\#\{k: \lambda_k=\lambda_n\},\\
    &{\max(u)=\max\{\sigma(k):\lambda_k=\lambda_n\},}
  \end{split}
\end{equation*}
and let
$$C(u)=\{x_{\sigma(1)},\dots,x_{\sigma(p)}\}
\cup \{x_{\sigma(k)}: p+1 \leq k \leq n-r,\ \sigma(k) < \max(u)\}.$$
The next result follows from Theorem \ref{HerTaka} and Equation
\eqref{eq_2.1}.

\begin{corollary}
  \label{cor:betti_shifted}
  Let $I$ be a symmetric shifted ideal. Then for
  all $i,j$ one has
$$\beta_{i,i+j}(I)=\sum_{u \in G(I),\ \deg u=j} \binom{|C(u)|}{i}.$$
\end{corollary}

\begin{example}
  Consider the ideal $I$ of Example \ref{exa:shifted_lin_quot}. Using
  Corollary \ref{cor:betti_shifted} and the information in Table
  \ref{tab:lin_quot_ex}, we obtain the following Betti table for $I$.
  \begin{equation*}
    \begin{matrix}
      &0&1&2&3\\\text{total:}&34&72&51&12\\\text{6:}&34&72&51&12\\
    \end{matrix}
  \end{equation*}
\end{example}

\section{Star configurations}
\label{sec:star-configurations}

In this section, we apply the results in the previous section to
symbolic powers of star configurations.  Recall that $I_{n,c}$ is the
monomial ideal of $S = \Bbbk [x_1,\dots,x_n]$ defined by
\begin{equation*}
  I_{n,c} = \bigcap_{1\leq i_1 < \dots < i_c \leq n}
  (x_{i_1}, \dots, x_{i_c})
\end{equation*}
and the $m$-th symbolic power of $I_{n,c}$ is given by
\begin{equation}
  \label{eq:symbolic_star}
  I^{(m)}_{n,c} = \bigcap_{1\leq i_1 < \dots < i_c \leq n}
  (x_{i_1}, \dots, x_{i_c})^m.
\end{equation}
We will show that $I_{n,c}^{(m)}$ is actually a symmetric strongly
shifted ideal.

\begin{proposition}
  \label{pro:gens_symbolic_star}
  For every integer $m\geq 1$, the ideal $I_{n,c}^{(m)}$ is
  $\sym_n$-fixed. Moreover
  \begin{equation*}
    \begin{split}
    &P(I_{n,c}^{(m)}) =
    \{\lambda \in P_n : |\lambda_{\leq c}|\geq m\},\\
    &\Lambda (I_{n,c}^{(m)}) =
    \{\lambda \in P_n : |\lambda_{\leq c}|= m,
    \forall i> c ~ \lambda_i = \lambda_c\}.
  \end{split}
  \end{equation*}
\end{proposition}

\begin{proof}
  Equation \eqref{eq:symbolic_star} immediately implies that
  $I_{n,c}^{(m)}$ is $\sym_n$-fixed because each element of $\sym_n$
  acts by permuting the primary components of the ideal.

  If $x^\lambda \in I_{n,c}^{(m)}$, then
  $x^\lambda \in (x_1,\dots,x_c)^m$, which gives
  $|\lambda_{\leq c}|\geq m$. Conversely, if
  $|\lambda_{\leq c}|\geq m$, then for all
  $1\leq i_1 < \dots < i_c \leq n$ and all $1\leq j\leq c$, we have
  $\lambda_j \leq \lambda_{i_j}$ because $\lambda$ is a
  partition. This implies
  \begin{equation*}
    \sum_{j=1}^c \lambda_{i_j} \geq \sum_{j=1}^c \lambda_j
    = |\lambda_{\leq c}|\geq m.
  \end{equation*}
  Hence, $x^\lambda \in (x_{i_1}, \dots, x_{i_c})^m$. Thus, the
  statement about $P (I_{n,c}^{(m)})$ is proved.

  Now the partition $\lambda = (\lambda_1, \dots, \lambda_n)$ is in
  $\Lambda (I_{n,c}^{(m)})$ if and only if
  $\lambda \in P (I_{n,c}^{(m)})$ and the partition obtained from
  $\lambda$ by decreasing any $\lambda_i$ is not in
  $P (I_{n,c}^{(m)})$. This forces $|\lambda_{\leq c}|= m$ and
  $\forall i> c$ $\lambda_i = \lambda_c$.
\end{proof}

\begin{example}
  \label{6:3:5}
  Figure \ref{fig:1} illustrates generators of $I_{n,c}^{(m)}$ when
  $n=6$, $c=3$ and $m=5$.

  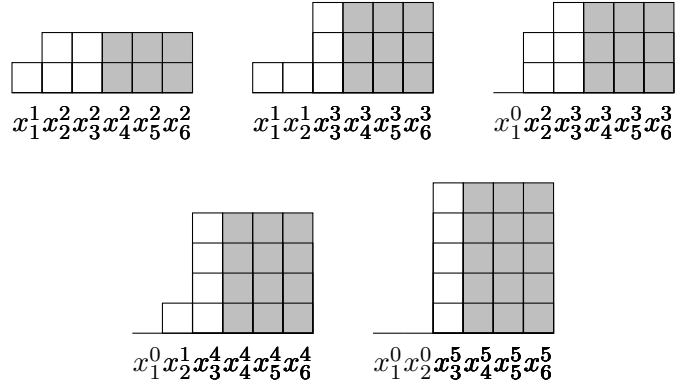
\begin{figure}[ht]
    \centering
    \begin{tikzpicture}[scale=0.4]
      \begin{scope}[shift={(0,8)}]
        \fill[lightgray] (4,1) rectangle (7,3);
        \foreach \p [count=\x] in {1,2,2,2,2,2}
        \foreach \y in {0,...,\p}
        {
          \draw (\x,1) rectangle (\x+1,\y+1);
          \draw (\x+0.5,0) node {$x_{\x}^{\p}$};
        }
      \end{scope}
      \begin{scope}[shift={(8,8)}]
        \fill[lightgray] (4,1) rectangle (7,4);
        \foreach \p [count=\x] in {1,1,3,3,3,3}
        \foreach \y in {0,...,\p}
        {
          \draw (\x,1) rectangle (\x+1,\y+1);
          \draw (\x+0.5,0) node {$x_{\x}^{\p}$};
        }
      \end{scope}
      \begin{scope}[shift={(16,8)}]
        \fill[lightgray] (4,1) rectangle (7,4);
        \foreach \p [count=\x] in {0,2,3,3,3,3}
        \foreach \y in {0,...,\p}
        {
          \draw (\x,1) rectangle (\x+1,\y+1);
          \draw (\x+0.5,0) node {$x_{\x}^{\p}$};
        }
      \end{scope}
      \begin{scope}[shift={(4,0)}]
        \fill[lightgray] (4,1) rectangle (7,5);
        \foreach \p [count=\x] in {0,1,4,4,4,4}
        \foreach \y in {0,...,\p}
        {
          \draw (\x,1) rectangle (\x+1,\y+1);
          \draw (\x+0.5,0) node {$x_{\x}^{\p}$};
        }
      \end{scope}
      \begin{scope}[shift={(12,0)}]
        \fill[lightgray] (4,1) rectangle (7,6);
        \foreach \p [count=\x] in {0,0,5,5,5,5}
        \foreach \y in {0,...,\p}
        {
          \draw (\x,1) rectangle (\x+1,\y+1);
          \draw (\x+0.5,0) node {$x_{\x}^{\p}$};
        }
      \end{scope}
    \end{tikzpicture}
    \caption{Partitions and monomials generating $I_{6,3}^{(5)}$.}
    \label{fig:1}
  \end{figure}
\end{example}

We now discuss the main result of this section.

\begin{theorem}
\label{thm:mainresult}
  For every integer $m\geq 1$, the ideal $I^{(m)}_{n,c}$ is strongly
  shifted.
\end{theorem}
\begin{proof}
  Let $\lambda = (\lambda_1, \dots, \lambda_n) \in P(I_{n,c}^{(m)})$
  and $1\leq k < l \leq n$ with $\lambda_k < \lambda_l$. Define
  $v = x^\lambda (x_k / x_l)$ and $\mu = \pa (v)$. We have
  $|\mu_{\leq c}| \geq |\lambda_{\leq c}|$. Thus, by Proposition
  \ref{pro:gens_symbolic_star}, we get $|\mu_{\leq c}| \geq m$ and
  $v \in I_{n,c}^{(m)}$.
\end{proof}

Using Corollary \ref{cor:betti_shifted}, we can derive several
consequences for the Betti tables of the ideals $I^{(m)}_{n,c}$.

\begin{corollary}
  \label{cor:partial_betti_star}\
  \begin{enumerate}
  \item The degree $j$ strand in the Betti table of $I_{n,c}^{(m)}$ is
    nonzero only when $j$ belongs to $\{m+k(n-c): k=1,2,\dots,m\}$,
    that is,
  $$\beta_{i,i+j}(I_{n,c}^{(m)})=0 \ \ \ \mbox{ for all $i\geq 0$ and $j \not \in \{m+k(n-c): k=1,2,\dots,m\}$}.$$
\item For all $i \geq 0$ and $k \geq \frac m 2 +1$, one has
  $$\beta_{i,i+m+k(n-c)}(I_{n,c}^{(m)})= \binom{c-1}{i}
  \beta_{0,m+k(n-c)}(I_{n,c}^{(m)}).$$ In particular,
  $\beta_{i,i+m+k(n-c)}(I_{n,c}^{(m)})$ only depends on the number of
  generators of $I_{n,c}^{(m)}$ of degree $m+k(n-c)$ when
  $k \geq \frac m 2 +1$.
\item The Castelnuovo-Mumford regularity of $I^{(m)}_{n,c}$ is
  $m (1+n-c)$. Moreover, if $m \geq 2$, then the bottom row in the
  Betti table of $I_{n,c}^{(m)}$ is given by the following formula:
  \begin{equation*}
    \beta_{i,i+m(1+n-c)}(I^{(m)}_{n,c})= \binom{n}{c-1} \binom{c-1}{i} \ \ \ \mbox{ for all $i\geq 0$}.
  \end{equation*}
\item If $m\leq c$, then
  \begin{equation*}
    \beta_{i,i+m+n-c}(I^{(m)}_{n,c})= \binom{n}{c-m-i} \binom{m+n-c+i-1}{i} \ \ \ \mbox{ for all $i\geq 0$}.
  \end{equation*}
\item All nonzero rows in the Betti table of $I^{(m)}_{n,c}$ have
  length $c-1$, with the exception of the top one.
\end{enumerate}
\end{corollary}

\begin{proof}
  It follows from Proposition \ref{pro:gens_symbolic_star}, that any
  element in $G (I_{n,c}^{(m)})$ has degree $m+k(n-c)$ with
  $1 \leq k \leq m$. Then Corollary \ref{cor:betti_shifted} proves
  (1).

Next we prove (2) and (3). Let $u \in G(I_{n,c}^{(m)})$ be a monomial
of degree greater than or equal to $m+(m/2+1)(n-c)$.  Corollary
\ref{cor:betti_shifted} says that, to prove (2), it is enough to
show that $|C(u)|=c-1$. By Proposition \ref{pro:gens_symbolic_star},
$u$ must be a monomial of the form
  $$u=\sigma(x_1^{\lambda_1} \cdots x_{c-1} ^{\lambda_{c-1}}x_c^{\lambda_c} x_{c+1}^{\lambda_c} \cdots x_n^{\lambda_c})$$
  with $\lambda_1+ \cdots + \lambda_c=m$, $\lambda_c\geq (m/2+1)$ and
  $\sigma \in \sym_n$. This says
  $\lambda_{c-1} \leq m/2-1 \leq \lambda_c-2$ and therefore
  $C(u)=\{x_{\sigma(1)},\dots,x_{\sigma(c-1)}\}$, as desired.  Then
  statement (3) follows from (2) together with the fact that the
  monomials of degree $m(1+m-c)$ in $G (I_{n,c}^{(m)})$ are precisely
  those in the set
  $\{\sigma(x_c^m x_{c+1}^m \cdots x_n^m): \sigma \in \sym_n\}.$

  To prove (4), let $m\leqslant c$.  It is easy to see that the minimum
  degree of monomials in $G (I_{n,c}^{(m)})$ is $m+n-c$. Also, the monomials of degree $m+n-c$ in $G (I_{n,c}^{(m)})$ are precisely those in the set
  \begin{equation*}
    X=\{\sigma(x_{c-m+1} x_{c-m+2} \dots x_n): \sigma \in \sym_n\}.
  \end{equation*}
  The monomials in $X$ are solely responsible for the top row in the
  Betti table of $I_{n,c}^{(m)}$. Note that $X$ is the set of all
  squarefree monomials of degree $m+n-c$. The Betti numbers of the
  ideal generated by $X$ can be found in\cite[Thm.\ 2.1]{Ga}, leading to the formula in (4).

  Finally, to see (5), let $q$ be the quotient of $m$ divided by
  $c$. Observe that for every $q+1 < k \leq m$, there is a monomial
  $u_k \in G(I_{n,c}^{(m)})$ with $\deg (u_k)=m+k(n-c)$ of the form
  $$u_k=x_1^{\lambda_1}\cdots x_{c-1}^{\lambda_{c-1}} x_c^k \cdots x_n^k$$
  with $\lambda_1 \leq \cdots \leq \lambda_{c-1} \leq k-1$. Then
  $C(u_k)=\{x_1,\dots,x_{c-1}\}$ and, by Corollary
  \ref{cor:betti_shifted}, row $m+k(n-c)$ in the Betti table of
  $I_{n,c}^{(m)}$ has length $c-1$.
\end{proof}

The previous corollary immediately gives a closed formula for the
Betti numbers of symbolic squares of star configurations. Note that
this result was first established in \cite[Theorem 3.2]{GHM}. Our
assumption about codimension eliminates degenerate cases where some
minimal generators disappear along with corresponding rows of the
Betti table.

\begin{corollary}
\label{cor:symbolic_square}
  If $c\geqslant 2$, then
  \begin{equation*}
    \beta_{i,i+j} (I^{(2)}_{n,c}) =
    \begin{cases}
      \binom{n}{c-2-i} \binom{n-c+1+i}{i},
      &j=n-c+2,\\
      \binom{n}{c-1} \binom{c-1}{i}, &j=2(n-c+1).
    \end{cases}
  \end{equation*}
\end{corollary}

Using Corollary \ref{cor:betti_shifted}, we can even give a closed
formula for the Betti numbers of the symbolic cube of star
configurations.

\begin{corollary}
  \label{cor:symbolic_cube}
  If $c\geqslant 3$, then
  \begin{equation*}
    \beta_{i,i+j} (I^{(3)}_{n,c}) =
    \begin{cases}
      \binom{n}{c-3-i} \binom{n-c+2+i}{i},
      &j=n-c+3,\\
      \binom{n}{c-2} \left(\binom{c-2}{i}+(n-c+1) \binom{c-1}{i}
      \right),
      &j=2(n-c+1)+1,\\
      \binom{n}{c-1} \binom{c-1}{i}, &j=3(n-c+1).
    \end{cases}
  \end{equation*}
\end{corollary}

\begin{proof}
  The top and bottom row of the Betti table are computed as in
  Corollary \ref{cor:partial_betti_star}.

  By Proposition \ref{pro:gens_symbolic_star}, the minimal generators
  of $I^{(3)}_{n,c}$ with degree $2(n-c+1)+1$ are the ones in the set
  \begin{equation*}
    \{\sigma(x_{c-1} x_{c}^2 x_{c+1}^2 \cdots x_n^2): \sigma \in \sym_n\}.
  \end{equation*}
  In particular, these are monomials $u=\sigma (x^\lambda)$ with
  \begin{equation*}
    \lambda = (\underbrace{0,\dots,0}_{c-2}, 1,
    \underbrace{2,\dots,2}_{n-c+1}),
  \end{equation*}
  and $p(\lambda) = c-2$, $r(\lambda) = n-c+1$. We also have
  $n-c+1 \leq \max (u)\leq n$. It follows that $|C(u)|=c-2$ if
  $\sigma (c-1) \geq \max (u)$, and $|C(u)|=c-1$ if
  $\sigma (c-1) < \max (u)$. We count how many monomials we have in
  each case.

  To produce a monomial $u$ with $|C(u)|=c-2$, we can first choose
  which variables have degree zero. Among the remaining variables, the
  single one having degree one must appear last. This can be
  accomplished in $\binom{n}{c-2}$ ways.

  To produce a monomial $u$ with $|C(u)|=c-1$, we can first choose
  which variables have degree zero. Next we can choose any one the
  remaining variables except the last one to appear with degree
  one. This can be accomplished in $\binom{n}{c-2} (n-c+1)$ ways.

  The statement now follows from Corollary \ref{cor:betti_shifted}.
\end{proof}

\begin{example}
  The Betti table of $I^{(3)}_{9,4}$ is
  \begin{equation*}
    \begin{matrix}
      &0&1&2&3\\
      \text{total:}&345&980&936&300\\
      \text{8:}&9&8&\text{.}&\text{.}\\
      \text{9:}&\text{.}&\text{.}&\text{.}&\text{.}\\
      \text{10:}&\text{.}&\text{.}&\text{.}&\text{.}\\
      \text{11:}&\text{.}&\text{.}&\text{.}&\text{.}\\
      \text{12:}&\text{.}&\text{.}&\text{.}&\text{.}\\
      \text{13:}&252&720&684&216\\
      \text{14:}&\text{.}&\text{.}&\text{.}&\text{.}\\
      \text{15:}&\text{.}&\text{.}&\text{.}&\text{.}\\
      \text{16:}&\text{.}&\text{.}&\text{.}&\text{.}\\
      \text{17:}&\text{.}&\text{.}&\text{.}&\text{.}\\
      \text{18:}&84&252&252&84\\
    \end{matrix}
  \end{equation*}
\end{example}

\begin{remark}
  It follows immediately from Corollary \ref{cor:symbolic_cube} that
  the third symbolic defect of $I_{n,c}$ is $\binom{n}{c-2}
  (n-c+2)$. We refer the reader to \cite{GGSVT18} for the definition
  of symbolic defect and \cite[Corollary 3.17]{GGSVT18} for a
  previously known bound.
\end{remark}

\begin{remark}
  \label{rem:fat_points}
  The ideal $I_{n,n-1}$ can be thought of as the defining ideal of the
  set of the $n$ points
  \begin{equation*}
    e_1 = [1:0:0:\ldots:0], e_2 = [0:1:0:\ldots:0],\dots,
    e_n = [0:0:0:\ldots:1] \in \mathbb{P}^{n-1}_\Bbbk
  \end{equation*}
  where $\mathbb{P}^{n-1}_\Bbbk$ denotes the $(n-1)$-dimensional
  projective space over $\Bbbk$. Similarly, $I_{n,n-1}^{(m)}$ can be
  thought of as the ideal defining the fat point scheme
  $m e_1 + m e_2 + \dots + m e_n$. For an introduction to fat points,
  we invite the reader to consult \cite{CH14}. If
  $\{p_1,\dots,p_n\} \subset \mathbb{P}^{n-1}_\Bbbk$ is a set of $n$
  points in general linear position, then there is a linear
  automorphism of $\mathbb{P}^{n-1}_\Bbbk$ taking $e_i$ to
  $p_i$. Algebraically, this corresponds to an invertible linear
  change of coordinates that preserves Betti numbers. In particular,
  it follows that the results of Corollary
  \ref{cor:partial_betti_star} provide information about the Betti
  numbers of the fat point scheme $m p_1 + m p_2 + \dots + m p_n$ in
  $\mathbb{P}^{n-1}_\Bbbk$. For more complete information, the Betti
  numbers of this fat point scheme can be computed by combining
  Proposition \ref{pro:gens_symbolic_star} and our later Corollary
  \ref{2.5}. We are grateful to Brian Harbourne for clarifying this
  connection.
\end{remark}

\section{Decompositions of symmetric shifted ideals}
\label{sec:decomp-symm-shift}

For $r \leqslant n$, let $J_{[n], r}$ be the ideal of
$S=\Bbbk[x_1,\dots,x_n]$ generated by all squarefree monomials of
degree $r$.  The ideal $J_{[n],r}$ is actually the same as the ideal
$I_{n,n+1-r}$ in the previous section, but we introduce a new notation
to simplify the proofs in Sections \ref{sec:decomp-symm-shift} and
\ref{sec:equiv-betti-numb}.  Note that $J_{[n], r}$ is $\sym_n$-fixed
and shifted. Its equivariant resolution has been described in
\cite[Theorem 4.11]{Ga}. In the following two sections we extend this
result to an arbitrary symmetric shifted ideal $I$ (see
Proposition~\ref{3.1}). This will be done in two steps. In this
section, we establish a coarse decomposition of
$\Tor_i(I,\Bbbk)_{i+d}$ (see Theorem~\ref{2.4}). This will be refined
in the next section.

We need some further notation.  Let $\Mon(S)$ be the set of all
monomials in $S$.  For monomial ideals $I \supset J$ of $S$, we write
$$\Mon(I/J)=\{u\in \Mon(S):u\in I,\ u \not \in J\}.$$
When both $I$ and $J$ are $\sym_n$-fixed, we write
$$\P(I/J)=\{\lambda \in \P_n: \lambda \in \P(I),\ \lambda \not \in \P(J)\}.$$
We note that $\Mon(I/J)$ is a $\Bbbk$-basis of $I/J$.

Let $\lambda=(\lambda_1,\dots,\lambda_n)$ be a partition and
$p=p(\lambda)$.  The following $S$-module $N^\lambda$ plays an
important role in our results:
$$N^\lambda=( \sigma(x^\lambda):\sigma \in \sym_n)/(\sigma(x_ix^\lambda): 1 \leq i\leq p, \sigma \in \sym_n).$$
We start by discussing some basic properties of the module
$N^\lambda$.  For $A=\{i_1,\dots,i_k\} \subset[n]$ with
$i_1<\cdots<i_k$, we write $x_A=x_{i_1} \cdots x_{i_k}$,
$\overline A=[n] \setminus A$, $\sym_A$ for the set of permutations on
$A$, $S_A=\Bbbk[x_{i_1},\dots,x_{i_k}]$,
$\mideal_A=(x_{i_1},\dots,x_{i_k}) \subset S_A$ the maximal ideal of
$S_A$, and $J_{A,r} \subset S_A$ the ideal of $S_A$ generated by all
squarefree monomials of degree $r$ in $S_A$.  For
$\mathbf a=(a_1,\dots,a_k)$, we write as above
$x_A^{\mathbf a}=x_{i_1}^{a_1} \cdots x_{i_k}^{a_k}$.

We set out to describe $\Mon (N^\lambda)$ starting with a preliminary example.

\begin{example}
  Let $\lambda = (0,1,1,2,2,3,3)$, so $p = p(\lambda) = 3$ and $r = r(\lambda) = 2$. In this case, $N^\lambda = I/J$ where
  \begin{equation*}
    \begin{split}
      I &= ( \sigma (x_1^0 x_2^1 x_3^1 x_4^2 x_5^2 x_6^3 x_7^3) : \sigma \in \sym_n),\\
      J &= ( \sigma (x_1^1 x_2^1 x_3^1 x_4^2 x_5^2 x_6^3 x_7^3) : \sigma \in \sym_n)
      + ( \sigma (x_1^0 x_2^1 x_3^2 x_4^2 x_5^2 x_6^3 x_7^3) : \sigma \in \sym_n).
    \end{split}
  \end{equation*}
  The monomial $x_2 x_3 x_4^2 x_5^2 x_6^3 x_7^3$ is an example of a
  monomial in $I$ but not in $J$. We can represent it as
  $m(x_4 x_5 x_6 x_7)^2 u$ where $m=x_2 x_3$ and $u=x_6 x_7$. This
  splits the indices of the variables into two sets: $A = \{1,2,3\}$
  and its complement $\overline{A} = \{4,5,6,7\}$. Note that
  $m = x_A^{\lambda_{\leqslant p}}$ is in the polynomial ring
  $S_A = \Bbbk [x_1,x_2,x_3]$, while $u$ is one of the minimal
  generators of $J_{\overline A, r}$ in
  $S_{\overline A} = \Bbbk [x_4,x_5,x_6,x_7]$. Also, the middle term
  $(x_{\overline A})^2 = (x_4 x_5 x_6 x_7)^2$ has exponent
  $\lambda_7 - 1$. Now notice that we can replace $m$ with any
  permutation $\sigma (x_A^{\lambda_{\leqslant p}})$ where
  $\sigma \in \sym_A$ and still obtain a monomial in $I$ and not in
  $J$; for example,
  \begin{equation*}
    x_1 x_2 x_4^2 x_5^2 x_6^3 x_7^3 = (x_1 x_2) (x_4 x_5 x_6 x_7)^2 u.
  \end{equation*}
  Similarly, we can replace
  $u$ by another generator (in fact, any monomial) of
  $J_{\overline A, r}$ and still obtain a monomial in
  $I$ and not in $J$; for example,
  \begin{equation*}
    \begin{split}
      &x_2 x_3 x_4^2 x_5^3 x_6^2 x_7^3 = m(x_4 x_5 x_6 x_7)^2 (x_5 x_7),\\
      &x_2 x_3 x_4^2 x_5^3 x_6^4 x_7^5 = m(x_4 x_5 x_6 x_7)^2 (x_5 x_6
      x_7^2).
    \end{split}
  \end{equation*}
  In addition, we could operate the same reasoning on any monomial
  obtained by permuting the variables in
  $x_2 x_3 x_4^2 x_5^2 x_6^3
  x_7^3$, leading to a similar split but with a different choice of
  index set $A$.  As we illustrate next, all elements of
  $\Mon
  (N^\lambda)$ can be obtained by combining these observations.
\end{example}

\begin{lemma}
  \label{2.1}
  Let $\lambda=(\lambda_1,\dots,\lambda_n)$ be a partition,
  $p=p(\lambda)$ and $r=r(\lambda)$.  Then
  \begin{equation}
    \label{2-1-1}
    \Mon(N^\lambda)= \biguplus_{A \subset [n],\ |A|=p} \left( \biguplus_{m \in \sym_A \cdot x_A^{\lambda_{\leq p}}} \{ m (x_{\overline A})^{\lambda_n-1} u : u \in \Mon(J_{\overline A,r})\} \right),
  \end{equation}
  and
  \begin{equation}
    \label{2-1-2}
    \P(N^\lambda)=\{(\lambda_1,\dots,\lambda_p,\mu_{p+1},\dots,\mu_n)\in \P_n: \mu_{p+1} \geq \lambda_n-1,\ \mu_{n-r+1} \geq \lambda_n\}.
  \end{equation}
\end{lemma}

\begin{proof}
  Equation \eqref{2-1-2} easily follows from \eqref{2-1-1}.  Hence, we
  only need to show \eqref{2-1-1}.  We first prove the inclusion
  ``$\subset$".  Let $u \in \Mon(N^\lambda)$. Then there is
  $\sigma \in \sym_n$ such that
  $\sigma(x^\lambda)=x_{\sigma(1)}^{\lambda_1} \cdots
  x_{\sigma(n)}^{\lambda_n}$ divides $u$.  We write
  $u=x_{\sigma(1)}^{a_1} \cdots x_{\sigma(n)}^{a_n}$.  Since $u$ is
  not contained in the ideal
  $(\sigma(x_ix^\lambda):1 \leq i \leq p, \sigma \in \sym_n)$, we have
  $a_1=\lambda_1,\dots,a_p=\lambda_p$.  Also, since
  $\sigma(x^\lambda)$ divides $u$, $a_k \geq \lambda_n-1$ for
  $p<k \leq n-r$, and $a_k \geq \lambda_n$ for $k \geq n-r+1$.  These
  inequalities imply that, by setting
  $A=\{\sigma(1),\dots,\sigma(p)\}$,
  $$u=\sigma(x_1^{\lambda_1} \cdots x_p^{\lambda_p}) (x_{\overline A})^{\lambda_{n}-1}(x_{\sigma(n-r+1)}\cdots x_{\sigma(n)})w$$
  for some $w \in \Mon(S_{\overline A})$, which shows that $u$ is
  contained in the right hand side of \eqref{2-1-1}.

  Next, we prove the inclusion ``$\supset$'' in \eqref{2-1-1}.  Let
  $u=m (x_{\overline A})^{\lambda_n-1} w$ with
  $m \in \sym_A \cdot x_A^{\lambda_{\leq p}}$ and
  $w \in \Mon(J_{\overline A,r})$.  By taking a permutation
  $\tau \in \sym_n$ appropriately,
  \begin{equation}
    \label{2-1}
    \tau(u)=x_1^{\lambda_1} \cdots x_p^{\lambda_p} x_{p+1}^{\lambda_n-1} \cdots x_n^{\lambda_n-1} \alpha
  \end{equation}
  with $\alpha \in \Mon(J_{\overline{[p]},r})$.  Moreover, we may
  choose $\tau$ so that $\alpha$ is divisible by
  $x_{n-r+1} \cdots x_n$. Then $x^\lambda$ divides $\tau(u)$ and
  $u \in (\sigma(x^\lambda):\sigma \in \sym_n)$.  We claim that $u$ is
  not contained in the ideal
  $J=(\sigma(x_i x^\lambda):1 \leq i \leq p, \sigma \in \sym_n)$.  We
  already see in \eqref{2-1} that if $\mu =\type(u)$ then $\mu$ is of
  the form
  $$\mu=(\lambda_1,\dots,\lambda_p,\mu_{p+1},\dots,\mu_n).$$
  Observe that
  $$\Lambda(J)=\{\lambda + e_i: 1 \leq i \leq p, \lambda+e_i \in \P_n\},$$
  where $e_i$ is the $i$-th standard basis vector of $\mathbb{Z}^n$.
  Since no element in $\Lambda(J)$ divides $\mu$, by Lemma \ref{1.1}
  the monomial $u \in \sym_n \cdot x^\mu$ is not contained in $J$.

  We finally show that the right-hand side of \eqref{2-1-1} is indeed
  a disjoint union.  To show this, it is enough to prove that for each
  $u=x_1^{a_1} \cdots x_n^{a_n}$ that is contained in the right-hand
  side of \eqref{2-1-1} there is a unique subset $A \subset [n]$ with
  $|A|=p$ and $m \in \sym_A \cdot x_A^{\lambda_{\leq p}}$ such that
  $u=m \alpha$ with $\alpha \in S_{\overline A}$.  Indeed, since
  $|\{ k: a_k<\lambda_n-1\}|=p$ by the shape of monomials in the
  right-hand side of \eqref{2-1-1}, such a set $A$ must be equal to
  the set $\{k : a_k <\lambda_n-1\}$, and a monomial $m$ must be
  $\prod_{i \in A} x_i^{a_i}$.
\end{proof}

Next, we decompose $N^\lambda$ into smaller modules which have a
simpler structure but are not fixed by the action of $\sym_n$.  Let
$\lambda \in \P_n$ be a partition, $p=p(\lambda)$ and $r=r(\lambda)$.
For $A \subset [n]$ with $|A|=p$ and
$m \in \sym_A \cdot x_A^{\lambda_{\leq p}}$, we define
$$N_{A,m}^\lambda
= \left( m \tau \left( x_{\overline
      A}^{(\lambda_{p+1},\dots,\lambda_n)} \right) :\tau \in
  \sym_{\overline A} \right)/ \left(x_i m \tau \left( x_{\overline
    A}^{(\lambda_{p+1},\dots,\lambda_n)} \right) :\tau \in
\sym_{\overline A},\ i \in A \right).$$ Recall that, for
$X \subset [n]$, $J_{X,r}$ is the monomial ideal of $S_X$ generated by
all squarefree monomials of degree $r$ in $S_X$ and
$\mideal_X=(x_i:i\in X)$ is the maximal ideal of $S_X$.  Since
$$(\lambda_{p+1},\dots,\lambda_n)=(\lambda_n-1,\dots,\lambda_n-1,\lambda_n,\dots,\lambda_n)$$
where $\lambda_n$ appears $r$ times, $N_{A,m}^\lambda$ is generated by
monomials
$\{m (x_{\overline A})^{\lambda_n-1} x_T:T \subset \overline A,
|T|=r\}$ and every monomial in $N^\lambda_{A,m}$ is divisible by
$m (x_{\overline A})^{\lambda_n-1}$.  Thus, by the map
$f \to f/(m (x_{\overline A})^{\lambda_n-1})$, we have an isomorphism
\begin{align}
  \label{2-2}
  \nonumber
  N_{A,m}^\lambda
  &\cong
    (x_T: T \subset\overline A, |T|=r)/(x_i x_T:T \subset \overline A,|T|=r, i \in A)\\
  & \cong
    \big(S/(x_i: i\in A)\big) \otimes_S (J_{\overline A,r} S)\\
  \nonumber
  &
    \cong
    S_A/\mideal_A \otimes_{\Bbbk}
    J_{\overline A,r}
\end{align}
where we consider that the last module is a module over
$S=S_A \otimes_{\Bbbk} S_{\overline A}$.

\begin{lemma}
  \label{2.2}
  Let $\lambda \in \P_n$ be a partition, $p=p(\lambda)$,
  $r=r(\lambda)$, $A \subset [n]$ with $|A|=p$, and
  $m \in \sym_A\cdot x^{\lambda_{\leqslant p}}_A$. Then
  \begin{itemize}
  \item[(i)]
    $\Mon(N_{A,m}^\lambda)=\{ m (x_{\overline A})^{\lambda_n-1} u: u
    \in \Mon(J_{\overline A,r})\}$.
  \item[(ii)] $N_{A,m}^\lambda$ is an $S$-submodule of $N^\lambda$.
  \item[(iii)]
    $\displaystyle N^\lambda= \bigoplus_{A \subset
      [n],|A|=p}\bigoplus_{m \in \sym_A \cdot x_A^{\lambda_{\leq p}}}
    N_{A,m}^\lambda$ (as $S$-modules).
  \end{itemize}
\end{lemma}

\begin{proof}
  Statement (i) follows from \eqref{2-2}.  To prove (ii), it is enough
  to show that for any $\{j_1,\dots,j_r\} \subset \overline A$, one
  has
  $$\operatorname{ann}_{N^\lambda}(m (x_{\overline A})^{\lambda_n-1} x_{j_1} \cdots x_{j_r})=(x_i:i \in A),$$
  where $\operatorname{ann}_M(h)=\{f \in S: hf=0\}$ for an $S$-module $M$
  and $h \in M$.  The inclusion ``$\supset$" is clear from the
  definition of $N^\lambda_{A,m}$.  To see the inclusion ``$\subset$",
  we must prove that for any monomial $u$ in $S_{\overline A}$,
  $m (x_{\overline A})^{\lambda_n-1} x_{j_1} \cdots x_{j_r} u$ is
  non-zero in $N^\lambda$, and this follows from Lemma \ref{2.1}.

  Statement (iii) follows from (i) and Lemma \ref{2.1}.
\end{proof}

\begin{corollary}
  \label{2.3}
  The $S$-modules $N_{A,m}^\lambda$ and $N^\lambda$ have linear
  resolutions.
\end{corollary}

\begin{proof}
  By the isomorphism in \eqref{2-2}, the tensor product of a minimal graded
  free resolution of $J_{\overline A,r}$ and one of $S_A/\mideal_A$ is
  isomorphic to a minimal graded free resolution of $N_{A,m}^\lambda$.  Since
  $J_{\overline A,r}$ and $S_A/\mideal_A$ have linear resolutions, the
  module $N_{A,m}^\lambda$ has a linear resolution.  Then $N^\lambda$
  also has a linear resolution by Lemma \ref{2.2}(iii).
\end{proof}

We now prove the main result of this section.

\begin{theorem}
  \label{2.4}
  If $I \subset S$ is a symmetric shifted ideal, then as
  $\Bbbk[\sym_n]$-modules we have
$$\Tor_i(I,\Bbbk)_{i+d} \cong \bigoplus_{\lambda \in \Lambda(I),\ |\lambda|=d} \Tor_i(N^\lambda,\Bbbk).
$$
\end{theorem}

To prove the above theorem, we first show the following statement.

\begin{lemma}
  \label{2.3.1}
  Let $I$ be a symmetric shifted ideal and
  $\Lambda(I)=\{\lambda^{(1)},\dots,\lambda^{(t)}\}$ with
  $\lambda^{(1)} <_\lex \cdots <_\lex \lambda^{(t)}$.  Let
  $I_{\leq k} \subseteq I$ be the $\sym_n$-fixed monomial ideal with
  $\Lambda(I_{\leq k})=\{\lambda^{(1)},\dots,\lambda^{(k)}\}$ for
  $k=1,2,\dots,t$.  Then, for $k=1,2,\dots,t$,
$$I_{\leq k}/I_{\leq k-1} \cong N^{\lambda^{(k)}}$$
as $S$-modules, where $I_{\leq 0}=(0)$.
\end{lemma}

\begin{proof}
  Observe that $I_{\leq k}$ is shifted and
  $$I_{\leq k}/I_{\leq k-1} =\left(\sigma(x^{\lambda^{(k)}}):\sigma \in \sym_n \right)/
  \left(( \sigma(x^{\lambda^{(k)}}):\sigma \in \sym_n ) \cap I_{\leq
      k-1}\right).$$ Since for monomial ideals $I,J,J'$ with
  $I \supset J$ and $I \supset J'$, we have $I/J=I/J'$ if and only if
  $\Mon(I/J)=\Mon(I/J')$, it is enough to prove
  $\P(I_{\leq k}/I_{\leq k-1})=\P(N^{\lambda^{(k)}})$.  Let
  $\lambda^{(k)}=(\lambda_1,\dots,\lambda_n)$, $p=p(\lambda^{(k)})$
  and $r=r(\lambda^{(k)})$.  Then we have
  \begin{align*}
    \P(I_{\leq k}/I_{\leq k-1})
    &= \P(I_{\leq k}) \setminus \P(I_{\leq k-1})\\
    &=
      \{\mu=(\mu_1,\dots,\mu_n) \in \P_n: \lambda^{(k)} \mbox{ divides } \mu, \lambda^{(k)}_{\leq p}=\mu_{\leq p}\}\\
    &=\{(\lambda_1,\dots,\lambda_p,\mu_{p+1},\dots,\mu_n) \in \P_n: \mu_{p+1} \geq \lambda_n-1,\ \mu_{n-r+1} \geq \lambda_n\}\\
    &=\P(N^{\lambda^{(k)}}),
  \end{align*}
  where we use Lemma \ref{1.3} for the third equality and Lemma
  \ref{2.1} for the last one.
\end{proof}

\begin{proof}[Proof of Theorem \ref{2.4}.]
  Let $\Lambda(I)=\{\lambda^{(1)},\dots,\lambda^{(t)}\}$ with
  $\lambda^{(1)} <_\lex \cdots <_\lex \lambda^{(t)}$ and let
  $I_{\leq k}$ be as in Lemma \ref{2.3.1}.  Then we have the short
  exact sequence
  \begin{equation}
    \label{2-3}
    0 \longrightarrow
    I_{\leq k-1} \longrightarrow I_{\leq k}
    \longrightarrow I_{\leq k}/I_{\leq k-1} \cong N^{\lambda^{(k)}}\longrightarrow 0.
  \end{equation}
  We prove
  $\Tor_i(I_{\leq k},\Bbbk) \cong \bigoplus_{l=1}^{k}
  \Tor_i(N^{\lambda^{(l)}},\Bbbk)$ using induction on $k$. Note that
  by Corollary \ref{2.3} this implies the desired statement.

  By the definition of the shifted property, the partition
  $\lambda^{(1)}$ must be a partition of the form
  $\lambda^{(1)}=(a,a,\dots,a,a+1,\dots,a+1)$.  Thus,
  $p(\lambda^{(1)})=0$ and
  $$I_{\leq 1}=( \sigma(x^{\lambda^{(1)}}):\sigma \in \sym_n)=N^{\lambda^{(1)}}.$$
  Hence, the assertion holds when $k=1$.

  Suppose $k>1$.  Since
  $|\lambda^{(1)}| \leq \cdots \leq |\lambda^{(t)}|$, using the
  inductive hypothesis we get
  $$\operatorname{reg}(I_{\leq k-1})= \max \{|\lambda^{(1)}|,\dots,|\lambda^{(k-1)}|\} \leq |\lambda^{(k)}|,$$
  because $N^{\lambda^{(l)}}$ is generated in degree $|\lambda^{(l)}|$
  and has a linear resolution by Corollary \ref{2.3}.  The short exact
  sequence in \eqref{2-3} induces the exact sequence
  \begin{align}
    \label{2-5-1}
    & \Tor_{i+1}(N^{\lambda^{(k)}},\Bbbk)_{i+1+(j-1)} \longrightarrow \Tor_i(I_{\leq k-1},\Bbbk)_{i+j}
      \longrightarrow\Tor_i(I_{\leq k},\Bbbk)_{i+j}\\
    \nonumber
    &\longrightarrow \Tor_i(N^{\lambda^{(k)}},\Bbbk)_{i+j} \longrightarrow \Tor_{i-1}(I_{\leq k-1},\Bbbk)_{i-1+(j+1)}.
  \end{align}
  Let $d=|\lambda^{(k)}|$.  By Corollary \ref{2.3},
  \begin{equation}
    \label{2-4}
    \Tor_i(N^{\lambda^{(k)}},\Bbbk)_{i+j}=0 \ \ \ \mbox{ for }j \ne d.
  \end{equation}
  Also, since $\operatorname{reg}(I_{\leq k-1})\leq d$,
  \begin{equation}
    \label{2-5}
    \Tor_{i-1}(I_{\leq k-1},\Bbbk)_{i-1+(j+1)}=0 \ \ \ \mbox{ for }j \geq d.
  \end{equation}
  Then \eqref{2-5-1}, \eqref{2-4} and \eqref{2-5} imply
  $$\Tor_i(I_{\leq k},\Bbbk)_{i+d} \cong \Tor_i(I_{\leq k-1},\Bbbk)_{i+d} \bigoplus \Tor_i(N^{\lambda^{(k)}},\Bbbk)_{i+d}$$
  and
  $$\Tor_i(I_{\leq k},\Bbbk)_{i+j} \cong \Tor_i(I_{\leq k-1},\Bbbk)_{i+j}\ \ \ \mbox{ for }j \ne d.$$
  These isomorphisms prove the desired statement.
\end{proof}

Using Theorem \ref{2.4}, it is possible to give a closed formula of
graded Betti numbers of a symmetric shifted ideal $I$ in terms of its
partition generator $\Lambda(I)$.  Let $\lambda \in \P_n$ with
$|\lambda|=d$, $p=p(\lambda)$ and $r=r(\lambda)$.  Then by \eqref{2-2}
we have
\begin{align*}
  \beta_{i,i+d}(N_{A,m}^\lambda)
  &= \beta_i(J_{\overline A,r} \otimes_{\Bbbk} (S_A/\mideal_A))\\
  &=\sum_{k+l=i} \beta_k(J_{\overline A,r}) \beta_l(S_A/\mideal_A)\\
  &=\sum_{k+l=i} \binom{n-p}{r+k} \binom{r+k-1}{k} \binom{p}{l},
\end{align*}
where we use the fact that
$$\beta_k(J_{\overline A,r}) =\binom{n-p}{r+k} \binom{r+k-1}{k}$$ (see,
e.g., \cite[Theorem 2.1]{Ga}).  For
$c=(c_1,c_2,\dots,c_n) \in \mathbb{Z}^n_{\geq 0}$, let
$c!= c_1! c_2! \cdots c_n!$.  For a partition
$\lambda=(\lambda_1,\dots,\lambda_n)$ with $|\lambda|=d \geq 0$, its
\textbf{type} $\type(c)=(t_0,t_1,\dots,t_d)$ is defined by
$t_i=|\{k:\lambda_k=i\}|$.  It is well-known that
$$|\sym_n \cdot x^\lambda|=\frac {n!} {\type(\lambda)!}.$$  Hence, by
Lemma \ref{2.2}(iii)
\begin{align*}
  \beta_i(N^\lambda)
  &=\binom{n}{p} \frac {p!} {\type(\lambda_{\leq p})!} \beta_i(N_{A,m}^\lambda)\\
  &= \sum_{k+l=i} \frac {p!} {\type(\lambda_{\leq p})!} \binom{n}{p} \binom{n-p}{r+k} \binom{r+k-1}{k} \binom{p}{l}.
\end{align*}
Thus, using Theorem \ref{2.4}, we obtain the following formula, which
may be viewed as a more explicit version of the formula given in
Corollary~\ref{cor:betti_shifted}.

\begin{corollary}
  \label{2.5}
  If $I \subset S$ is a symmetric shifted ideal, then
  $$\beta_{i,i+d}(I)
  =\sum_{\lambda \in \Lambda(I), |\lambda|=d} \left( \sum_{k+l=i}
    \frac {p(\lambda)!} {\type(\lambda_{\leq p(\lambda)})!}
    \binom{n}{p(\lambda)} \binom{n-p(\lambda)}{r(\lambda)+k}
    \binom{r(\lambda) + k-1}{k} \binom{p(\lambda)}{l}\right).$$
\end{corollary}

\begin{example}
  Let $I=J^{(2)}_{[n],r} (=I^{(2)}_{n,n+1-r})$. Then $I$ is generated
  by two partitions $\lambda=(0^{n-r-1},1^{r+1})$ and
  $\mu =(0^{n-r},2^r)$, where $a^i$ denotes
  $(a,a,\dots,a) \in \mathbb Z^i$.  In this case, $p(\lambda)=0,$
  $r(\lambda)=r+1$, $p(\mu)=n-r$ and $r(\mu)=r$.  Corollary \ref{2.5}
  says
  $$
  \beta_{i,i+r+1}(I)= \sum_{k+l=i} \frac {0!} {0!} \binom{n}{0}
  \binom{n}{r+1+k} \binom{r+k}{k} \binom{0}{l} = \binom{n}{r+1+i}
  \binom{r+i}{i}$$ and
  $$
  \beta_{i,i+2r}(I)= \sum_{k+l=i} \frac {(r+1)!} {(r+1)!}
  \binom{n}{n-r} \binom{n-n+r}{r+k} \binom{r+k-1}{k} \binom{n-r}{l} =
  \binom{n}{n-r} \binom{n-r}{i}.
  $$
  This recovers Corollary \ref{cor:symbolic_square}.
\end{example}

\section{Equivariant Betti numbers}
\label{sec:equiv-betti-numb}

While Corollary \ref{2.5} gives a closed formula for the graded Betti
numbers of symmetric shifted ideals, the formula is not simple.  To
understand these numbers better, we refine the decomposition of
$\Tor_i(I,\Bbbk)$ given in Theorem~\ref{2.4}.  In this section, we
give an explicit description of the $\Bbbk[\sym_n]$-module structure of
$\Tor_i(I,\Bbbk)$ for a symmetric shifted ideal $I$ by using Theorem
\ref{2.4}, and explain how it helps to determine Betti numbers of
these ideals by examples.  We refer the reader to \cite{Sagan} for
some basics on representation theory, such as induced representations
and Specht modules.

For a monomial $m \in S_A$, let $\M_A(m)=\operatorname{span}_{\Bbbk}\{\sigma(m):\sigma \in \sym_A\}$.
We denote by $\operatorname{Ind}_{\sym_k \times \sym_l}^{\sym_{k+l}}(K \boxtimes K')$ the induced representation of the tensor product of a $\Bbbk[\sym_k]$-module $K$ and a $\Bbbk[\sym_l]$-module $K'$.
Let $I$ be a symmetric shifted ideal.
By Theorem \ref{2.4}, we know  $\Tor_i(I,\Bbbk)\cong \bigoplus_{\lambda \in \Lambda(I)} \Tor_i(N^\lambda,\Bbbk)$.
Thus, to understand the $\Bbbk[\sym_n]$-module structure of $\Tor_i(I,\Bbbk)$ it is enough to consider the $\Bbbk[\sym_n]$-module structure of $\Tor_i(N^\lambda,\Bbbk)$.

Let $\lambda \in \P_n$, $p=p(\lambda)$ and $r=r(\lambda)$.  For each
subset $A\subset [n]$ with $|A|=p$, fix a permutation
$\rho_A\in \sym_n$ such that $\rho_A ([p])=A$. The set
$\{\rho_A \in \sym_n: A \subset[n], |A|=p\}$ is a set of
representatives of $\sym_n/(\sym_p \times \sym_{n-p})$.

By Lemma
\ref{2.2}(iii) and \eqref{2-2}, we have an isomorphism (up to shift of
degrees)
\begin{align}
  \nonumber
  N^\lambda
  &= \bigoplus_{A \subset [n],\ |A|=p}
    \bigoplus_{m \in \sym_A \cdot x_A^{\lambda_{\leq p}}} N_{A,m}^\lambda\\
  \nonumber
  & \cong \bigoplus_{A \subset [n],\ |A|=p}
    \left( \bigoplus_{m \in \sym_A \cdot x_A^{\lambda_{\leq p}}} \left( m (S_A/\mideal_A) \otimes_\Bbbk (x_{\overline A})^{\lambda_n-1} J_{\overline A,r} \right) \right)\\
  \nonumber
  & \cong \bigoplus_{A \subset [n],\ |A|=p}
    \left[ \left( \M_A(x_A^{\lambda_{\leq p}}) \otimes_\Bbbk (S_A/\mideal_A) \right) \otimes_\Bbbk (x_{\overline A})^{\lambda_n-1} J_{\overline A,r} \right]\\
  \nonumber
  & \cong \bigoplus_{A \subset [n],\ |A|=p}
    \rho_A \left[ \left( \M_{[p]}(x_{[p]}^{\lambda_{\leq p}}) \otimes_\Bbbk (S_{[p]}/\mideal_{[p]}) \right) \otimes_\Bbbk \left((x_{p+1}\cdots x_n)^{\lambda_n-1} J_{\overline {[p]},r} \right) \right]\\
  \nonumber
  & \cong \operatorname{Ind}_{\sym_p \times \sym_{n-p}}^{\sym_n} \left[\left( \M_{[p]}(x^{\lambda_{\leq p}}_{[p]}) \otimes_{\Bbbk} S_{[p]}/\mideal_{[p]}\right) \boxtimes ((x_{p+1}\cdots x_n)^{\lambda_n-1} J_{\overline{[p]},r})\right]\\
  \label{3-1}
  & \cong \operatorname{Ind}_{\sym_p \times \sym_{n-p}}^{\sym_n} \left[\left( \M_{[p]}(x^{\lambda_{\leq p}}_{[p]}) \otimes_{\Bbbk} S_{[p]}/\mideal_{[p]}\right) \boxtimes J_{\overline{[p]},r} \right],
\end{align}
where
$\M_{[p]}(x^{\lambda_{\leq p}}_{[p]})\otimes_{\Bbbk}
S_{[p]}/\mideal_{[p]} =\Bbbk$ if $p=0$.  Hence, we conclude that
$N^\lambda$ is isomorphic to the module \eqref{3-1} as
$\Bbbk[\sym_n]$-modules.  Note that as an $S_{[p]}$-module,
$\M_{[p]}(x^{\lambda_{\leq p}}_{[p]}) \otimes_\Bbbk
S_{[p]}/\mideal_{[p]}$ is the direct sum of
$|\sym_p \cdot x^{\lambda_{\leq p}}_{[p]}|$ copies of
$S_{[p]}/\mideal_{[p]}$.  Recall that, for an $S_A$-module $N$ and an
$S_{\overline A}$-module $M$, there is an isomorphism
$$\Tor_i^S(N \otimes_\Bbbk M,\Bbbk)\cong \bigoplus_{k+l=i} \Tor_k^{S_A}(N,\Bbbk) \otimes_\Bbbk \Tor_l^{S_{\overline A}}(M,\Bbbk).$$
Then the decomposition in \eqref{3-1} shows that we have an
isomorphism of $\Bbbk[\sym_n]$-modules
$$\Tor^S_i(N^\lambda,\Bbbk) \cong \bigoplus_{k+l=i} \left[
  \operatorname{Ind}_{\sym_p \times \sym_{n-p}}^{\sym_n} \left(
    \M_{[p]} (x^{\lambda_{\leq p}}_{[p]}) \otimes_\Bbbk
    \Tor^{S_{[p]}}_k(S_{[p]}/\mideal_{[p]}) \right) \boxtimes \left(
    \Tor^{S_{\overline{[p]}}}_l (I_{\overline{[p]},r},\Bbbk)
  \right)\right].
$$
Let $S^\lambda$ be the Specht module associated to the partition
$\lambda=(\lambda_1,\dots,\lambda_p)$ with $\lambda_1>0$ (see, e.g.,
\cite[\S 2.3]{Sagan} or \cite[\S 3]{Ga}).  For an integer $l \geq p$,
set
$$U_{l}^\lambda= \operatorname{Ind}_{\sym_p\times\sym_{l-p}}^{\sym_l} S^\lambda \boxtimes S^{(l-p)}.$$
Galetto \cite[Corollary 4.12]{Ga} proved
\begin{equation}
  \label{EqBettiSq}
  \Tor_i^{S_{[n]}}(J_{[n],r},\Bbbk) \cong U_{n}^{(1^i,r)}
\end{equation}
as $\Bbbk[\sym_n]$-modules.  This says
$$\Tor_i^{S_{[p]}}(S_{[p]}/\mideal_{[p]},\Bbbk) \boxtimes
\Tor^{S_{\overline{[p]}}}_l (J_{\overline{[p]},r},\Bbbk) \cong
U_p^{(1^i)} \boxtimes U_{n-p}^{(1^i,r)}
$$
as $\Bbbk[\sym_{p} \times \sym_{n-p}]$-modules.  Combining all these
facts, we get the following.

\begin{proposition}
  \label{3.1}
  Let $\lambda \in \P_n$, $p=p(\lambda)$ and $r=r(\lambda)$.  As
  $\Bbbk[\sym_n]$-modules,
  $$\Tor_i(N^\lambda,\Bbbk) \cong
  \bigoplus_{k+l=i} \left( \operatorname{Ind}_{\sym_p \times
      \sym_{n-p}}^{\sym_n} \left( \M_{[p]}(x^{\lambda_{\leq p}}_{[p]})
      \otimes_\Bbbk U_p^{(1^k)} \right) \boxtimes U_{n-p}^{(1^l,r)}
  \right).$$
\end{proposition}

We note that $\M(x^\lambda)$ is isomorphic to a $\Bbbk[\sym_n]$-module
known as a permutation module \cite[\S 2.1]{Sagan}.

\begin{theorem}
  \label{3.2}
  Let $I$ be a symmetric shifted ideal.  Then as
  $\Bbbk[\sym_n]$-modules
  $$\Tor_i(I,\Bbbk)_{i+d}
  \cong \bigoplus_{\substack{\lambda \in \Lambda(I) \\ |\lambda|=d}}
  \bigoplus_{k+l=i}\left( \operatorname{Ind}_{\sym_{p(\lambda)} \times
      \sym_{n-p(\lambda)}}^{\sym_n} \left(
      \M_{[p(\lambda)]}(x^{\lambda_{\leq p(\lambda)}}_{[p (\lambda)]})
      \otimes_\Bbbk U_{p(\lambda)}^{(1^k)} \right) \boxtimes
    U_{n-p(\lambda)}^{(1^l,r(\lambda))} \right).$$
\end{theorem}

In the rest of this section, we explain how Theorem \ref{3.2} is
useful to write down Betti numbers of symmetric shifted ideals.  To do
this, we identify $S^\lambda$ with the Ferrers diagram corresponding
to partition $\lambda$.  Also, for simplicity, we write
$$\operatorname{Ind}_{\sym_p \times \sym_{n-p}}^{\sym_n} N \boxtimes M=
N \boxtimes M \text{ and }
\operatorname{Ind}_{\sym_p \times \sym_{n-p}}^{\sym_n} N \boxtimes
S^{(n-p)}=N_{\uparrow n}.$$  By Theorem \ref{3.2}, the
$\Bbbk[\sym_n]$-module structure of $\Tor(N^\lambda,\Bbbk)$ only
depends on $p(\lambda),r(\lambda)$ and $\lambda_{\leq p(\lambda)}$. We
write
$$\operatorname{info}(\lambda)=(p(\lambda),r(\lambda),\lambda_{\leq
  p(\lambda)}).$$

\begin{example}
  Let $I=J_{[n],r} \subset \Bbbk[x_1,\dots,x_n]$ be the monomial ideal
  generated by all squarefree monomials of degree $r$.  As we already
  mentioned in \eqref{EqBettiSq}, we have
  $$\Tor_i(I,\Bbbk) \cong U_n^{(1^i,r)}$$
  for all $i$. Here we check that our formula in Theorem \ref{3.2}
  coincides with this.  In this case, $\Lambda(I)=\{(0^{n-r},1^r)\}$.
  Let $\lambda=(0^{n-r},1^r)$.  Then since
  $\operatorname{info}(\lambda)=(0,r,\emptyset)$, we have
  $$\bigoplus_{k+l=i} \left( \M_{[p(\lambda)]}(x^{\lambda_{\leq p(\lambda)}})
    \otimes_\Bbbk U_{p(\lambda)}^{(1^k)} \right) \boxtimes
  U_{n-p(\lambda)}^{(1^l,r(\lambda))} = U_{n}^{(1^i,r)}$$ and
  Theorem \ref{3.2} yields
  $$\Tor_i(I,\Bbbk) \cong \Tor_i(N^\lambda,\Bbbk)
  \cong U_n^{(1^i,r)}.$$
\end{example}

Graded Betti numbers of an $S$-module $N$ are often presented by a
Betti table, i.e., the table whose $(i,j)$-th entry is
$\beta_{i,i+j}(N)$.

For a module $N^\lambda$ and a symmetric shifted
ideal, we present their graded Betti numbers by the table whose
$(i,j)$-th entry is the $\Bbbk[\sym_n]$-module given in Theorem
\ref{3.2}.  We call such table an \textbf{equivariant Betti table}.

For example, the equivariant Betti table of $I_{6,3}$ is
\begin{center}
  \medskip{}
  \begin{tabular}{c|ccccccccc}
    & 0 & 1 & 2 & 3\\
    \hline
    \\
    3 & $\ydiagram{3}_{\uparrow 6}$ & $\ydiagram{1,3}_{\uparrow 6}$ & $\ydiagram{1,1,3}_{\uparrow 6}$
                & $\ydiagram{1,1,1,3}$
  \end{tabular}
  \medskip{}
\end{center}

\begin{example}
  \label{ex2}
  Let $I=J_{[n],r}^{(2)}$ be the second symbolic power of the
  squarefree Veronese ideal with $n\geq r+1$.  Then
  $\Lambda(I)=\{\lambda, \mu\}$, where $\lambda=(0^{n-r-1},1^{r+1})$
  and $\mu=(0^{n-r},2^r)$.  Using
  $\operatorname{info}(\lambda)=(0,r+1,\emptyset)$ and
  $\operatorname{info}(\mu)=(n-r,r,(0^{n-r}))$, we obtain
  $$\Tor_i(N^\lambda,\Bbbk)
  \cong U_n^{(1^i,r+1)}$$ and
  $$
  \Tor_i(N^{\mu},\Bbbk) \cong \bigoplus_{k+l=i} \left(\left(
      \M_{[n-r]}(x^{(0^{n-r})}) \otimes U_{n-r}^{(1^k)}\right) \boxtimes
    U_r^{(1^l,r)}\right) =U_{n-r}^{(1^i)} \boxtimes U_r^{(r)}
  =(U_{n-r}^{(1^i)})_{\uparrow n}.$$
  The equivariant Betti table of
  $N^\lambda$ and $N^\mu$ when $n=6$ and $r=3$ are:
  \begin{center}
    \medskip{}
    \begin{tabular}{c|ccccccccc}
      $N^\lambda$& 0 & 1 & 2 \\
      \hline
      \\
      4 & $\ydiagram{4}_{\uparrow 6}$ & $\ydiagram{1,4}_{\uparrow 6}$ & $\ydiagram{1,1,4}$
    \end{tabular}
    \bigskip{}

    \begin{tabular}{c|ccccccccc}
      $N^\mu$ & 0 & 1 & 2 & 3\\
      \hline
      \\
      6 & $\left(\emptyset_{\uparrow 3}\right)_{\uparrow 6}$ & $\left(\ydiagram{1}_{\uparrow 3}\right)_{\uparrow 6}$ & $\left(\ydiagram{1,1}_{\uparrow 3}\right)_{\uparrow 6}$
                      & $\ydiagram{1,1,1}_{\uparrow 6}$
    \end{tabular}
    \medskip{}
  \end{center}
  The equivariant Betti table of $I$ is given by the sum of the two
  tables above as follows:
  \begin{center}
    \medskip{}
    \begin{tabular}{c|ccccccccc}
      $I$ & 0 & 1 & 2 & 3\\
      \hline
      \\
      4 & $\ydiagram{4}_{\uparrow 6}$ & $\ydiagram{1,4}_{\uparrow 6}$ & $\ydiagram{1,1,4}$
      \\
      \\
      6 & $\left(\emptyset_{\uparrow 3}\right)_{\uparrow 6}$ & $\left(\ydiagram{1}_{\uparrow 3}\right)_{\uparrow 6}$ & $\left(\ydiagram{1,1}_{\uparrow 3}\right)_{\uparrow 6}$
                      & $\ydiagram{1,1,1}_{\uparrow 6}$
      \\
    \end{tabular}
    \medskip{}
  \end{center}
\end{example}

\begin{example}
  \label{ex3}
  Let $I=I_{n,r}^{(3)}$ be the third symbolic power of the squarefree
  Veronese ideal with $n\geq r+2$.  Then
  $\Lambda(I)=\{\lambda,\mu,\rho\}$ with
  $\lambda=(0^{n-r-2},1^{r+2}), \mu=(0^{n-r-1},1,2^r),\rho=(0^{n-r},3^r).$
  Using that
  $\operatorname{info}(\lambda)=(0,r+2,\emptyset),
  \operatorname{info}(\mu)=(n-r-1,r,(0^{n-r-1}))$ and
  $ \operatorname{info}(\rho)=(n-r,r,(0^{n-r}))$, we have
  \begin{align*}
    \Tor_i(N^\lambda,\Bbbk) &\cong U_n^{(1^i,r+2)}\\
    \Tor_i(N^\mu,\Bbbk) &\cong \bigoplus_{k+l=i}  (\M_{[n-r-1]}(x^{(0^{n-r-1})}) \otimes U_{n-r-1}^{(1^k)} )\boxtimes U_{r+1}^{(1^l,r)} \\
                            &= \left( U_{n-r-1}^{(1^i)} \boxtimes U_{r+1}^{(r)} \right) \bigoplus \left ( U_{n-r-1}^{(1^{i-1})} \boxtimes U_{r+1}^{(1,r)} \right),\\
    \Tor_i(N^\rho,\Bbbk) & \cong \bigoplus_{k+l=i}
                           (\M_{[n-r]}(x^{(0^{n-r})}) \otimes U_{n-r}^{(1^k)} )\boxtimes U_{r}^{(1^l,r)}
                           = U_{n-r}^{(1^i)} \boxtimes U_r^{(r)}=(U_{n-r}^{(1^i)})_{\uparrow n}.
  \end{align*}
  The equivariant Betti table of $I$ is the sum of the equivariant
  Betti table of $N^\lambda$, $N^\mu$ and $N^\rho$. The following
  tables are the equivariant Betti tables of these three modules when
  $n=6$ and $r=3$.
  \begin{center}
    \medskip{}
    \begin{tabular}{c|ccccccccc}
      $N^\lambda$& 0 & 1 \\
      \hline
      \\
      5 & $\ydiagram{5}_{\uparrow 6}$ & $\ydiagram{1,5}$
    \end{tabular}
    \bigskip{}

    \begin{tabular}{c|ccccccccc}
      $N^\mu$ & 0 & 1 & 2 & 3\\
      \hline
      \\
              &
                $\emptyset_{\uparrow 2} \boxtimes \ydiagram{3}_{\uparrow 4}$
                  & $\ydiagram{1}_{\uparrow 2} \boxtimes \ydiagram{3}_{\uparrow 4}$
                      & $\ydiagram{1,1} \boxtimes \ydiagram{3}_{\uparrow 4}$ \\
      7 & & $\oplus$ & $\oplus$
      \\&&
           $ \emptyset_{\uparrow 2} \boxtimes \ydiagram{1,3}$
                  &
                    $ \ydiagram{1}_{\uparrow 2} \boxtimes \ydiagram{1,3}$
                      &
                        $ \ydiagram{1,1} \boxtimes \ydiagram{1,3}$
    \end{tabular}
    \bigskip{}

    \begin{tabular}{c|ccccccccc}
      $N^\rho$ & 0 & 1 & 2 & 3\\
      \hline
      \\
      9 & $\left(\emptyset_{\uparrow 3}\right)_{\uparrow 6}$ & $\left(\ydiagram{1}_{\uparrow 3}\right)_{\uparrow 6}$ & $\left(\ydiagram{1,1}_{\uparrow 3}\right)_{\uparrow 6}$
                           & $\ydiagram{1,1,1}_{\uparrow 6}$
    \end{tabular}
    \medskip{}
  \end{center}
\end{example}

\begin{example}
  \label{ex4}
  Let $I=(x_1,\dots,x_n)^3 \subset \Bbbk[x_1,\dots,x_n]$ with
  $n \geq 3$. Then $\Lambda(I)=\{\lambda,\mu,\rho\}$ with
  $\lambda=(0^{n-3},1^3),\mu=(0^{n-2},1,2),\rho=(0^{n-1},3).$ A
  computation similar to Example \ref{ex4} shows that the equivariant
  Betti tables of $N^\lambda$, $N^\mu$ and $N^\rho$ are \hspace{30pt}
  \begin{center}
    \medskip{}
\begin{tabular}{c|ccccccccc}
$N^\lambda$& 0 & 1 & 2 & $\cdots$ & $n-2$ & $n-1$ \\
\hline
\\
3 & $\ydiagram{3}_{\uparrow n}$ & $\ydiagram{1,3}_{\uparrow n}$ &  $\ydiagram{1,1,3}_{\uparrow n}$ &
 $\cdots$ &
$\begin{ytableau}
\\
\none[]\vspace{-5pt}
\\
\none[\vdots]\\
\\
 &  &
\end{ytableau}
_{\uparrow n}$ &$\begin{ytableau}
\\
\none[]\vspace{-5pt}\\
\none[\vdots]\\
\\
\\
 &  &
\end{ytableau}$
\end{tabular}
\bigskip{}

\begin{tabular}{c|ccccccccc}
$N^\mu$ & 0 & 1 & 2 & $\cdots$ & $n-2$ & $n-1$\\
\hline
\\
&
$\emptyset_{\uparrow n-2} \boxtimes \ydiagram{1}_{\uparrow 2}$
& $\ydiagram{1}_{\uparrow n-2} \boxtimes \ydiagram{1}_{\uparrow 2}$
& $\ydiagram{1,1}_{\uparrow n-2} \boxtimes \ydiagram{1}_{\uparrow 2}$
& $\cdots$ &
$\begin{ytableau}
\\
\none[]\vspace{-5pt}\\
\none[\vdots]\\
\\
\\
\end{ytableau} \boxtimes \ydiagram{1}_{\uparrow 2}$\\
3&&$\oplus$&$\oplus$&&$\oplus$
\\
&&
$\emptyset_{\uparrow n-2} \boxtimes \ydiagram{1,1}$
&
$\ydiagram{1}_{\uparrow n-2} \boxtimes \ydiagram{1,1}$
&
$\cdots$
&
$
\begin{ytableau}
\\
\none[]\vspace{-5pt}\\
\none[\vdots]\\
\\
\end{ytableau}_{\uparrow n-2} \hspace{-15pt} \boxtimes \ydiagram{1,1}_{\uparrow 2}$
&
$
\begin{ytableau}
\\
\none[]\vspace{-5pt}\\
\none[\vdots]\\
\\
\\
\end{ytableau} \boxtimes \ydiagram{1,1}$
\end{tabular}
\bigskip{}

\begin{tabular}{c|ccccccccc}
$N^\rho$ & 0 & 1 & 2 & $\cdots$ &  $n-1$\\
\hline
\\
3 & $(\emptyset_{\uparrow n-1})_{\uparrow n}$ & $(\ydiagram{1}_{\uparrow n-1})_{\uparrow n}$ & $\left(\ydiagram{1,1}_{\uparrow n-1}\right)_{\uparrow n}$
& $\cdots$ &
$\begin{ytableau}
\\
\\
\none[]\vspace{-5pt}\\
\none[\vdots]
\\
\\
\end{ytableau}_{\uparrow n}$
\end{tabular}
\medskip{}
\end{center}
\end{example}

\begin{example}
  Let $I=(I_{n,2})^2$ with $n \geq 4$.  Then
  $\Lambda(I)=\{\lambda,\mu,\rho\}$ with
  $$\lambda=(0^{n-4},1^4),\mu=(0^{n-3},1^2,2),\rho=(0^{n-2},2^2)$$
  and
  \begin{align*}
    \Tor_i(N^\lambda,\Bbbk) &\cong U_n^{(1^i,4)},\\
    \Tor_i(N^\mu,\Bbbk) &\cong \bigoplus_{k+l=i} \left(\M_{[n-3]} (x^{(0^{n-3})})\otimes U_{n-3}^{(1^k)} \right) \boxtimes U_3^{(1^{l+1})}\\
                            & \cong
                              \left(U_{n-3}^{(1^i)} \boxtimes U_3^{(1)} \right)
                              \bigoplus
                              \left(U_{n-3}^{(1^{i-1})} \boxtimes U_3^{(1^2)} \right)
                              \bigoplus
                              \left(U_{n-3}^{(1^{i-2})} \boxtimes U_3^{(1^3)} \right),
    \\
    \Tor_i(N^\rho,\Bbbk) &\cong U_{n-2}^{(1^i)} \boxtimes U_2^{(2)}=(U_{n-2}^{(1^i)})_{\uparrow n}.
  \end{align*}
\end{example}

\section{Other considerations}
\label{sec:other-considerations}

\subsection*{Weakly polymatroidal ideals}
Our definition of symmetric shifted ideals is inspired by stable
monomial ideals, which also have linear quotients (see \cite[\S
7]{HH}), but almost all stable monomial ideals are not fixed by an
action of the symmetric group.
Besides stable monomial ideals, another
famous class of monomial ideals which have linear quotients are
(weakly) polymatroidal ideals (see \cite[\S 12]{HH} for more details).
A monomial ideal $I \subset S$ is said to be \textbf{weakly
  polymatroidal} if for any two monomials
$u=x_1^{a_1} \cdots x_n^{a_n}$ and
$v=x_1^{b_1}\cdots x_n^{b_n} \in G(I)$ such that
$a_1=b_1, \dots,a_{t-1}=b_{t-1}$ and $a_t>b_t$ for some $t$, there is
$j>t$ such that $v(x_t/x_j) \in I$.

One may wonder whether
$I^{(m)}_{n,c}$ is a weakly polymatroidal ideal and the fact that it
has linear quotients follows from the weakly polymatroidal property.
The next example shows this is not the case.

\begin{example}
  Consider the ideal $I=I_{6,3}^{(5)}$ which we also studied in
  Example \ref{6:3:5}.  Recall that this ideal is generated by the
  $\sym_6$-orbits of the following five monomials
  \begin{equation}
    \label{generators6:3:5}
    x_1x_2^2x_3^2x_4^2x_5^2x_6^2,\quad
    x_1x_2x_3^3x_4^3x_5^3x_6^3,\quad
    x_2^2x_3^3x_4^3x_5^3x_6^3,\quad
    x_2x_3^4x_4^4x_5^4x_6^4,\quad
    x_3^5x_4^5x_5^5x_6^5.
  \end{equation}
  Then the two monomials
  $$u=x_1^{a_1} \cdots x_6^{a_6}=x_1^7x_2^4x_3^4x_4^4x_5^1x_6^0 \mbox{ and }
  v=x_1^{b_1} \cdots x_6^{b_6}=x_1^5x_2^5x_3^5x_4^5x_5^0x_6^0$$ are
  contained in $I$.  Clearly $a_1>b_1$, but for any $j>1$ the monomial
  $v(x_1/x_j)$ must belong to the $\sym_6$-orbit of
  $x_3^4x_4^5x_5^5x_6^6$.  However, the monomial
  $x_3^4x_4^5x_5^5x_6^6$ is not divisible by any monomial listed in
  \eqref{generators6:3:5}, so $I$ is not weakly polymatroidal.
\end{example}

\subsection*{Open questions}
Finally, we give a few open problems relating to symmetric shifted
ideals.  We give a formula for (equivariant) Betti numbers of
symmetric shifted ideals, but we could not construct their minimal
graded free resolutions.  On the other hand, an explicit $\sym_n$-equivariant
minimal graded free resolutions of $I_{n,c}$ is constructed in \cite{Ga}.

\begin{problem}
  Construct explicit $\sym_n$-equivariant minimal graded free resolutions of
  symmetric shifted ideals.
\end{problem}

Symmetric shifted ideals give a class of $\sym_n$-fixed monomial
ideals having linear resolutions. However, we do not know if there is
an $\sym_n$-fixed monomial ideal which is not shifted but has a linear
resolution.  This prompts the following:

\begin{problem}
\label{last_problem}
  Find a combinatorial characterization of $\sym_n$-fixed monomial
  ideals having linear resolutions.
\end{problem}

\begin{remark}
  After this paper was posted on arXiv, Claudiu Raicu \cite{Ra19} gave
  an answer to Problem \ref{last_problem}.  He proves that if an
  $\sym_n$-fixed monomial ideal has a linear resolution then it must
  be a symmetric shifted ideal.  In particular, Theorem \ref{thm:1}
  and his result imply that an $\sym_n$-fixed monomial ideal has
  linear quotients if and only if it is a symmetric shifted ideal.
\end{remark}


\begin{thebibliography}{9999}

\bibitem[AS12]{AS12}
J.~Ahn and Y.~S.~Shin.
\newblock The minimal free resolution of a star-configuration in {$\mathbb{P}^n$}
  and the weak {L}efschetz property.
\newblock {\em J. Korean Math. Soc.}, 49(2):405--417, 2012.

\bibitem[AS14]{AS14}
J.~Ahn and Y.~S.~Shin.
\newblock The minimal free resolution of a fat star-configuration in {$\mathbb{P}^n$}.
\newblock {\em Algebra Colloq.}, 21(1):157--166, 2014.



\bibitem[AH07]{AH07}
M.~Aschenbrenner and C.~J.~Hillar,
\newblock Finite generation of symmetric ideals.
\newblock {\em Trans. Amer. Math. Soc.}, 359(11):5171--5192, 2007.

\bibitem[BDRH+09]{BDRH+}
T.~Bauer, S.~Rocco, B.~Harbourne, M.~Kapustka, A.~Knutsen, W.~Syzdek, and T.~Szemberg.
\newblock A primer on {S}eshadri constants.
\newblock In {\em Interactions of classical and numerical algebraic geometry},
  volume 496 of {\em Contemp. Math.}, pages 33--70. Amer. Math. Soc.,
  Providence, RI, 2009.

\bibitem[BH10]{BH10}
C.~Bocci and B.~Harbourne.
\newblock Comparing powers and symbolic powers of ideals.
\newblock {\em J. Algebraic Geom.}, 19(3):399--417, 2010.

\bibitem[CCG08]{CCG08}
E.~Carlini, L.~Chiantini, and A.~V.~Geramita.
\newblock Complete intersections on general hypersurfaces.
\newblock {\em Michigan Math. J.}, 57:121--136, 2008.
\newblock Special volume in honor of Melvin Hochster.

\bibitem[CGVT14]{CGVT14}
E.~Carlini, E.~Guardo, and A.~Van Tuyl.
\newblock Star configurations on generic hypersurfaces.
\newblock {\em J. Algebra}, 407:1--20, 2014.

\bibitem[CH14]{CH14}
S.~M.~Cooper and B.~Harbourne.
\newblock Regina lectures on fat points.
\newblock In {\em Connections between algebra, combinatorics, and geometry},
  volume~76 of {\em Springer Proc. Math. Stat.}, pages 147--187. Springer, New
  York, 2014.

\bibitem[CHT11]{CHT11}
S.~M.~Cooper, B.~Harbourne, and Z.~Teitler.
\newblock Combinatorial bounds on {H}ilbert functions of fat points in
  projective space.
\newblock {\em J. Pure Appl. Algebra}, 215(9):2165--2179, 2011.

\bibitem[CLO07]{CLO}
D.~Cox, J.~Little, and D.~O'Shea.
\newblock {\em Ideals, varieties, and algorithms}.
\newblock Undergraduate Texts in Mathematics. Springer, New York, third
  edition, 2007.
\newblock An introduction to computational algebraic geometry and commutative
  algebra.

\bibitem[Dr14]{Dr14}
J.~Draisma,
\newblock Noetherianity up to symmetry.
\newblock
In {\em Combinatorial algebraic geometry},
volume 2108 of {\em Lecture Notes in Mathematics}, pages 33--61, Springer, 2014.

\bibitem[EK90]{EK}
S.~Eliahou and M.~Kervaire,
Minimal resolutions of some monomial ideals,
\textit{J. Algebra} \textbf{129} (1990),
1--25.

\bibitem[Fra08]{Fra08}
C.~A.~Francisco.
\newblock Tetrahedral curves via graphs and {A}lexander duality.
\newblock {\em J. Pure Appl. Algebra}, 212(2):364--375, 2008.

\bibitem[FMN06]{FMN06}
C.~A.~Francisco, J.~C.~Migliore, and U.~Nagel.
\newblock On the componentwise linearity and the minimal free resolution of a
  tetrahedral curve.
\newblock {\em J. Algebra}, 299(2):535--569, 2006.

\bibitem[Ga18]{Ga} F.~Galetto, On the ideal generated by all
  squarefree monomials of a given degree, \textit{J. Commut. Algebra},
  to appear,  {\href{https://projecteuclid.org/euclid.jca/1523433695}{\nolinkurl{https://projecteuclid.org/euclid.jca/1523433695}}}.

\bibitem[GGSVT18]{GGSVT18}
F.~Galetto, A.~V.~Geramita, Y.~S. Shin, and A.~Van Tuyl.
\newblock The symbolic defect of an ideal.
\newblock {\em J. Pure Appl. Algebra}, 2018.
\url{https://doi.org/10.1016/j.jpaa.2018.11.019}

\bibitem[GHM13]{GHM}
A.~V.~Geramita, B.~Harbourne, and J.~Migliore.
\newblock Star configurations in {$\mathbb{P}\sp n$}.
\newblock {\em J. Algebra}, 376:279--299, 2013.

\bibitem[GHMN15]{GHMN}
A.~V. Geramita, B.~Harbourne, J.~Migliore, and U.~Nagel.
\newblock Matroid configurations and symbolic powers of their ideals, Trans.\ Amer.\
Math.\ Soc.\  {\bf 369} (2017), 7049--7066.

\bibitem[GMS06]{GMS06}
A.~V. Geramita, J.~Migliore, and L.~Sabourin.
\newblock On the first infinitesimal neighborhood of a linear configuration of
  points in {$\mathbb{P}^2$}.
\newblock {\em J. Algebra}, 298(2):563--611, 2006.

\bibitem[GS]{M2}
D.~R. Grayson and M.~E. Stillman.
\newblock Macaulay2, a software system for research in algebraic geometry.
\newblock Available at \url{http://www.math.uiuc.edu/Macaulay2/}.

\bibitem[HH11]{HH}
J.~Herzog and T.~Hibi,
Monomial ideals,
Graduate Texts in Mathematics, vol.\ 260, Springer-Verlag, 2011.

\bibitem[HT02]{HT}
J.~Herzog and Y.~Takayama,
\newblock Resolutions by mapping cones,
\newblock {\em Homology Homotopy Appl.} \textbf{4} (2002), 277–-294.

\bibitem[HS12]{HS12}
C.~J.~Hillar and S.~Sullivant,
\newblock Finite Gr\"{o}bner bases in infinite dimensional polynomial rings and applications.
\newblock {\em Adv. Math.}, 229(1):1--25, 2012).

\bibitem[LNNR18a]{LNNR18a}
D.~V.~Le, U.~Nagel, H.~D.~Nguyen, and T.~R\"{o}mer,
\newblock Castelnuovo-Mumford regularity up to symmetry.
\newblock  {\em Int. Math. Res. Not.} (to appear), \texttt{arXiv:1806.00457}.

\bibitem[LNNR18b]{LNNR18b}
D.~V.~Le, U.~Nagel, H.~D.~Nguyen, and T.~R\"{o}mer,
\newblock Codimension and Projective Dimension up to Symmetry.
\newblock {\em Math. Nachr.} (to appear), \texttt{arXiv:1809.06877}.

\bibitem[LS19]{LS19}
K.-N.~Lin and Y.-H.~Shen.
\newblock Symbolic powers and free resolutions of generalized star
  configurations of hypersurfaces, 2019, arXiv:1912.04448.

\bibitem[Man19]{1907.08172}
P.~Mantero.
\newblock The structure and free resolutions of the symbolic powers of star
configurations of hypersurfaces.
\newblock Preprint, 2019, \texttt{arXiv:1907.08172}.

\bibitem[MN05]{MN05}
J.~Migliore and U.~Nagel.
\newblock Tetrahedral curves.
\newblock {\em Int. Math. Res. Not.}, (15):899--939, 2005.

\bibitem[NR17]{NR17}
U.~Nagel and T.~R\"{o}mer,
\newblock Equivariant Hilbert series in non-Noetherian polynomial rings.
\newblock {\em J. Algebra}, 486:204--245, 2017.



\bibitem[PS15]{PS15}
J.~P.~Park and Y.~S. Shin.
\newblock The minimal free graded resolution of a star-configuration in
  {$\mathbb{P}\sp n$}.
\newblock {\em J. Pure Appl. Algebra}, 219(6):2124--2133, 2015.

\bibitem[Ra19]{Ra19}
C.~Raicu,
Regularity of $\mathfrak S_n$-invariant monomial ideals,
Preprint, 2019, \texttt{arXiv:1909.04650}. 

\bibitem[Sch82]{Sch82}
P.~W.~Schwartau.
\newblock {\em Liaison Addition and Monomial Ideals}.
\newblock ProQuest LLC, Ann Arbor, MI, 1982.
\newblock Thesis (Ph.D.)--Brandeis University.

\bibitem[Sa01]{Sagan}
B.~Sagan:
\newblock The symmetric group. Representations, combinatorial algorithms, and symmetric functions. Second edition.
\newblock Graduate Texts in Mathematics, vol.\ 203, Springer-Verlag, 2001.



\bibitem[Sta99]{Stanley}
R.~P.~Stanley.
\newblock {\em Enumerative combinatorics. {V}ol. 2}, volume~62 of {\em
  Cambridge Studies in Advanced Mathematics}.
\newblock Cambridge University Press, Cambridge, 1999.
\newblock With a foreword by Gian-Carlo Rota and appendix 1 by Sergey Fomin.

\bibitem[TVAV11]{TVAV11}
D.~Testa, A.~V\'{a}rilly-Alvarado, and M.~Velasco.
\newblock Big rational surfaces.
\newblock {\em Math. Ann.}, 351(1):95--107, 2011.

\bibitem[Var11]{Var11}
M.~Varbaro.
\newblock Symbolic powers and matroids.
\newblock {\em Proc. Amer. Math. Soc.}, 139(7):2357--2366, 2011.

\bibitem[ZS75]{ZS}
O.~Zariski and P.~Samuel.
\newblock {\em Commutative algebra. {V}ol. {II}}.
\newblock Springer-Verlag, New York-Heidelberg, 1975.
\newblock Reprint of the 1960 edition, Graduate Texts in Mathematics, Vol. 29.
\end{thebibliography}
\end{document}